\documentclass[11pt,reqno]{article}
\usepackage{mathrsfs}
\usepackage{amsmath,amssymb,amsthm}
\newtheorem{theorem}{Theorem}[section]
\newtheorem{lemma}[theorem]{Lemma}

\newtheorem{corollary}[theorem]{Corollary}
\newtheorem{definition}[theorem]{Definition}
\newtheorem{remark}[theorem]{Remark}
\usepackage[english]{babel}
\numberwithin{equation}{section}

\title{\bf Sobolev-Lorentz spaces with an application to the inhomogeneous biharmonic NLS equation}
\author{{JinMyong An, PyongJo Ryu, JinMyong Kim$^*$}\\
\footnotesize{Faculty of Mathematics, {\bf Kim Il Sung} University, Pyongyang, Democratic People's Republic of Korea}\\
\footnotesize{$^*$ Corresponding Author: JinMyong Kim (jm.kim0211@ryongnamsan.edu.kp)}}
\date{}
\begin{document}
\maketitle
\begin{abstract}
We consider the Cauchy problem for the inhomogeneous biharmonic nonlinear Schr\"{o}dinger (IBNLS) equation
\[iu_{t} +\Delta^{2} u=\lambda |x|^{-b}|u|^{\sigma}u,\;u(0)=u_{0} \in H^{s} (\mathbb R^{d}),\]
where $\lambda\in \mathbb R$, $d\in \mathbb N$, $0\le s<\min\left\{2+\frac{d}{2},d\right\}$, $0<b<\min \left\{4,\; d-s,\; 2+\frac{d}{2}-s \right\}$ and $0<\sigma\le \sigma_{c}(s)$ with $\sigma<\infty$. Here $\sigma_{c}(s)=\frac{8-2b}{d-2s}$ if $s<\frac{d}{2}$, and $\sigma_{c}(s)=\infty$ if $s\ge \frac{d}{2}$.
First, we give some remarks on Sobolev-Lorentz spaces and extend the chain rule under Lorentz norms for the fractional Laplacian $(-\Delta)^{s/2}$ with $s\in (0,1]$ established by \cite{AT21} to any $s>0$. Applying this estimate and the contraction mapping principle based on Strichartz estimates in Lorentz spaces, we then establish the local well-posedness in $H^{s}$ for the IBNLS equation in both of subcritical case $\sigma<\sigma_{c}(s)$ and critical case $\sigma=\sigma_{c}(s)$. We also prove that the IBNLS equation is globally well-posed in $H^{s}$, if the initial data is sufficiently small and $\frac{8-2b}{d}\le \sigma\le \sigma_{c}(s)$ with $\sigma<\infty$.
\end{abstract}


\section{Introduction}\label{sec 1.}

In this paper, we study the Cauchy problem for the inhomogeneous biharmonic nonlinear Schr\"{o}dinger (IBNLS) equation
\begin{equation} \label{GrindEQ__1_1_}
\left\{\begin{array}{l} {iu_{t} +\Delta^{2}u=\lambda |x|^{-b} |u|^{\sigma} u,~(t,x)\in \mathbb R\times \mathbb R^{d},}\\
{u(0,x)=u_{0}(x) \in H^{s}(\mathbb R^{d})}, \end{array}\right.
\end{equation}
where $d\in \mathbb N$, $s\ge 0$, $0<b<4$, $\sigma>0$ and $\lambda \in \mathbb R$.
The limiting case $b=0$ (classic biharmonic nonlinear Schr\"{o}dinger equation) was introduced by Karpman \cite{K96} and Karpman-Shagalov \cite{KS97} to take into account the role of small fourth-order dispersion terms in the propagation of intense laser beams in a bulk medium with Kerr nonlinearity and it has attracted a lot of interest during the last two decades. See, for example, \cite{D18B,D18I,D21,GC07,LZ21,MZ16,P07,PX13} and the references therein.
The local and global well-posedness as well as scattering and blow-up in the energy space $H^{2}$ have been widely studied.
See \cite{D21,MZ16,P07,PX13} and the references for example.
On the other hand, the local and global well-posedness in the fractional Sobolev spaces $H^s$ have also been studied by several authors. See \cite{D18B,D18I,GC07,LZ21} for example.

The equation \eqref{GrindEQ__1_1_} has a counterpart for the Laplacian operator, namely, the inhomogeneous nonlinear Schr\"{o}dinger (INLS) equation
\begin{equation} \label{GrindEQ__1_2_}
iu_{t} +\Delta u=\lambda |x|^{-b} |u|^{\sigma} u.
\end{equation}
The INLS equation \eqref{GrindEQ__1_2_} arises in nonlinear optics for modeling the propagation of laser beam and it has been widely studied by many authors. See, for example, \cite{AT21,AK211,AK212,AKC21,AK213,C21,DK21,G17,GM21,MMZ21} and the references therein.
The local and global well-posedness as well as blow-up and scattering in the energy space $H^{1}$ have been widely studied by many authors.
See, for example, \cite{C21,DK21,GM21, MMZ21} and the references therein.
Meanwhile, the local and global well-posedness for the INLS equation \eqref{GrindEQ__1_2_} in the fractional Sobolev space have also been studied by \cite{AT21,AK211,AK212,AKC21,AK213,G17}.

The IBNLS equation \eqref{GrindEQ__1_1_} is invariant under scaling $u_{\alpha}(t,x)=\alpha^{\frac{4-b}{\sigma}}u(\alpha^{4}t,\alpha x ),~\alpha >0$.
An easy computation shows that
\begin{equation} \label{GrindEQ__1_3_}
\left\|u_{\alpha}(t)\right\|_{\dot{H}^{s}}=\alpha^{s+\frac{4-b}{\sigma}-\frac{d}{2}}\left\|u(t)\right\|_{\dot{H}^{s}}.
\end{equation}
We thus define the critical Sobolev index
\begin{equation} \label{GrindEQ__1_4_}
s_{c}:=\frac{d}{2}-\frac{4-b}{\sigma}.
\end{equation}
Putting
\begin{equation} \label{GrindEQ__1_5_}
\sigma_{c}(s):=
\left\{\begin{array}{cl}
{\frac{8-2b}{d-2s},} ~&{{\rm if}~s<\frac{d}{2},}\\
{\infty,}~&{{\rm if}~s\ge \frac{d}{2},}
\end{array}\right.
\end{equation}
we can easily see that $s>s_{c}$ is equivalent to $\sigma<\sigma_{c}(s)$. If $s<\frac{d}{2}$, then $s=s_{c}$ is equivalent to $\sigma=\sigma_{c}(s)$.
For initial data $u_{0}\in H^{s}(\mathbb R^{d})$, we say that the Cauchy problem \eqref{GrindEQ__1_1_} is $H^{s}$-critical (for short, critical) if $0\le s<\frac{d}{2}$ and $\sigma=\sigma_{c}(s)$.
If $s\ge 0$ and $\sigma<\sigma_{c}(s)$, then the problem \eqref{GrindEQ__1_1_} is said to be  $H^{s}$-subcritical (for short, subcritical).
Especially, if $\sigma =\frac{8}{d-2s}$, then the problem is known as $L^{2}$-critical or mass-critical.
If $\sigma =\frac{8-2b}{d-4}$ with $d\ge 5$, it is called $H^{2}$-critical or energy-critical.
Throughout the paper, a pair $(\gamma(p),p)$ is said to be biharmonic Schr\"{o}dinger admissible or $B$-admissible if
\begin{equation} \label{GrindEQ__1_6_}
\left\{\begin{array}{ll}
{2\le p\le\frac{2d}{d-4}},~&{{\rm if}~d>4,}\\
{2\le p<\infty,}~&{{\rm if}~d\le 4,}
\end{array}\right.
\end{equation}
and
\begin{equation} \label{GrindEQ__1_7_}
\frac{4}{\gamma (p)} =\frac{d}{2} -\frac{d}{p} .
\end{equation}
The IBNLS equation \eqref{GrindEQ__1_1_} also enjoys the conservations of mass and energy, which are defined respectively by
\begin{equation} \label{GrindEQ__1_8_}
M\left(u(t)\right):=\int_{\mathbb R^{d}}{| u(t,x)|^{2}dx}=M\left(u_{0} \right),
\end{equation}
\begin{equation} \label{GrindEQ__1_9_}
E\left(u(t)\right):=\frac{1}{2}\int_{\mathbb R^{d}}{|\Delta u(t,x)|^2 dx}-\frac{\lambda
}{\sigma+2}\int_{\mathbb R^{d}}{|x|^{-b}\left|u(t,x)\right|^{\sigma+2}dx}=E\left(u_{0} \right).
\end{equation}

The IBNLS equation \eqref{GrindEQ__1_1_} has attracted a lot of interest in recent years. See, for example, \cite{ARK221, ARK222,CG21, CGP20, GP20, GP21, LZ212, S21} and the references therein.
Guzm\'{a}n-Pastor \cite{GP20} proved that \eqref{GrindEQ__1_1_} is locally well-posed in $L^{2}$, if $d\in \mathbb N$, $0<b<\min\left\{4,d\right\}$ and $0<\sigma<\sigma_{c}(0)$. They also established the local well-posedness in $H^{2}$ for $d\ge 3$, $0<b<\min\{\frac{d}{2},4\}$, $\max\{0,\frac{2-2b}{d}\}<\sigma<\sigma_{c}(2)$.
Furthermore, they obtained the global well-posedness results in $H^{2}$ for $\max\{0,\frac{2-2b}{d}\}<\sigma\le \frac{8-2b}{d}$.
The global well-posedness and scattering in $H^{2}$ in the intercritical case $\frac{8-2b}{d}<\sigma<\sigma_{c}(2)$ were also studied in \cite{CG21, CGP20, GP21, S21}.
Afterwards, Cardoso-Guzm\'{a}n-Pastor \cite{CGP20} established the local and global well-posedness of \eqref{GrindEQ__1_1_} in $\dot{H}^{s_{c}}\cap \dot{H}^{2}$ with $d\ge 5$, $0<s_{c}<2$, $0<b<\min\{\frac{d}{2},4\}$ and $\max\{1,\frac{8-2b}{d}\}<\sigma< \frac{8-2b}{d-4}$.
Recently, Liu-Zhang \cite{LZ212} established the local well-posedness in $H^{s}$ with $0<s<2$ by using the Besov space theory.
More precisely, they proved that the IBNLS equation \eqref{GrindEQ__1_1_} is locally well-posed in $H^{s}$ if $d\in \mathbb N$, $0<s\le 2$, $0<b<\min\{\frac{d}{2},4\}$ and $0<\sigma<\sigma_{c}(s)$. See Theorem 1.5 of \cite{LZ212} for details.
This result about the local well-posedness of \eqref{GrindEQ__1_1_} improves the one of \cite{GP20} by not only extending the validity of $d$ and $s$ but also removing the lower bound $\sigma>\frac{2-2b}{d}$.
They also obtained the global well-posedness result in $H^{2}$ in the full range of mass-subcritical case and mass-critical cases $0<\sigma\le \frac{8-2b}{d}$.
Very recently, the authors in \cite{ARK221, ARK222} established the local and global well-posedness in $H^{s}$ for the IBNLS equation \eqref{GrindEQ__1_1_} with $d\in \mathbb N$, $0\le s <\min \left\{2+\frac{d}{2},\frac{3}{2}d\right\}$, $0<b<\min\left\{4,d,\frac{3}{2}d-s,\frac{d}{2}+2-s\right\}$ and $0<\sigma<\sigma_{c}(s)$.

As mentioned above, the local and global well-posedness in $H^{s}$ for the IBNLS equation \eqref{GrindEQ__1_1_} have been widely studied for the $H^{s}$-subcritical case, i.e. $\sigma<\sigma_{c}(s)$. However, up to the knowledge of the authors, the local and global well-posedness for the IBNLS equation \eqref{GrindEQ__1_1_} in the $H^{s}$-critical case, i.e. $\sigma=\frac{8-2b}{d-2s}$ with $s<\frac{d}{2}$ was not known until very recently.

The main purpose of this paper is to establish the local and global well-posedness in $H^{s}$ with $0\le s<\min\left\{2+\frac{d}{2},d\right\}$ for the IBNLS equation \eqref{GrindEQ__1_1_} in both of subcritical case $\sigma<\sigma_{c}(s)$ and critical case $\sigma=\sigma_{c}(s)$.
To arrive at this goal, we give some remarks on Sobolev-Lorentz spaces and extend the chain rule under Lorentz norms for the fractional Laplacian $(-\Delta)^{s/2}$ with $s\in (0,1]$ established by \cite{AT21} to any $s>0$. We then establish the various nonlinear estimates and use the contraction mapping principle based on Strichartz estimates in Lorentz spaces.

The first main result of this paper concerns with the local well-posedness for the IBNLS equation \eqref{GrindEQ__1_1_} in $H^{s}(\mathbb R^{d})$ with $0\le s<\min\left\{2+\frac{d}{2},d\right\}$ .
\begin{theorem}\label{thm 1.1.}
Let $d\in \mathbb N$, $0\le s<\min\left\{2+\frac{d}{2},d\right\}$ , $0<b<\min \left\{4,\; d-s,\; 2+\frac{d}{2}-s \right\}$ and $0<\sigma\le \sigma_{c}(s)$ with $\sigma<\infty$. If $\sigma$ is not an even integer, assume further
\footnote[1]{For $s\in \mathbb R$, $\left\lceil s\right\rceil$ denotes the minimal integer which is larger than or equal to $s$} $\sigma>\left\lceil s\right\rceil-1$.
Then for any $u_{0}\in H^{s}(\mathbb R^{d}) $, there exist $T_{\max }=T_{\max }(u_{0})>0$ and $T_{\min }=T_{\min }(u_{0})>0$ such that \eqref{GrindEQ__1_1_} has a unique, maximal solution satisfying
\begin{equation} \label{GrindEQ__1_10_}
u\in C\left(\left(-T_{\min } ,\;T_{\max } \right),\;H^{s} \right)\cap L^{\gamma (q)}_{\rm loc} \left(\left(-T_{\min } ,\;T_{\max } \right),\;H_{q,2}^{s} \right),
\end{equation}
for any $B$-admissible pair $\left(\gamma (q),q\right)$. If $T_{\max}<\infty$ (resp. if $T_{\min}<\infty$), then $\left\|u_{0}\right\|\rightarrow \infty$ as $t\uparrow T_{\max}$ (resp. as $t\downarrow -T_{\min}$).
Moreover, the solution of \eqref{GrindEQ__1_1_} depends continuously on the initial data $u_{0}$ in the following sense. There exists $0<T<T_{\max } ,\, T_{\min } $ such that if $u_{0}^{m} \to u_{0} $ in $H^{s}$ and if $u_{m} $ denotes the solution of \eqref{GrindEQ__1_1_} with the initial data $u_{0}^{m} $, then $0<T<T_{\max } \left(u_{0}^{m} \right),\, T_{\min } \left(u_{0}^{m} \right)$ for all sufficiently large $m$ and $u_{m} \to u$ in $L^{\gamma (q)} \left(\left[-T,\;T\right],\;L^{q,2}\right)$ as $m\to \infty $ for any $B$-admissible pair $\left(\gamma (q), q\right)$. Especially, if $s>0$, then $u_{m} \to u$ in $C\left(\left[-T,\;T\right],\; H^{s-\varepsilon } \right)$ for all $\varepsilon >0$.
\end{theorem}
\begin{remark}\label{rem 1.3.}
\textnormal{In Theorem \ref{thm 1.1.}, the result for the $H^{s}$-critical case $\sigma=\sigma_{c}(s)$ with $0\le s<\frac{d}{2}$ is completely new. The result for the $H^{s}$-subcritical case was already generalized in our work \cite{AK211} to $0\le s <\min \left\{2+\frac{d}{2},\frac{3}{2}d\right\}$ and $0<b<\min\left\{4,d,\frac{3}{2}d-s,\frac{d}{2}+2-s\right\}$. However, in this paper, we give the simple and unified proof in both of critical and subcritical case. Moreover, we obtain more precise regularity result with respect to known ones, since $H_{q,2}^{s}$ is a strict subspace of $H_{q}^{s}$.}
\end{remark}
The second main result of this paper is about the global well-posedness and scattering for the IBNLS equation \eqref{GrindEQ__1_1_} with small initial data in $H^{s}$ with $0\le s<\min\left\{2+\frac{d}{2},d\right\}$ .
\begin{theorem}\label{thm 1.3.}
Let $d\in \mathbb N$, $0\le s<\min\left\{2+\frac{d}{2},d\right\}$, $0<b<\min \left\{4,\; d-s,\; 2+\frac{d}{2}-s \right\}$ and $\frac{8-2b}{d}\le \sigma\le \sigma_{c}(s)$ with $\sigma<\infty$. If $\sigma$ is not an even integer, assume further $\sigma>\left\lceil s\right\rceil-1$. Let
\begin{equation} \label{GrindEQ__1_11_}
s_{c}=\frac{d}{2}-\frac{4-b}{\sigma},~\tilde{s}_{c}=\frac{d}{2}-\frac{d-2s+8-2b}{2(\sigma+1)}.
\end{equation}
Then for any $u_{0} \in H^{s}(\mathbb R^{d})$ satisfying $\left\|u_{0}\right\|_{\dot{H}^{s_{c}}\cap \dot{H}^{\tilde{s}_{c}}}<\delta$ for some $\delta>0$ small enough, there exists a unique, global solution of \eqref{GrindEQ__1_1_} satisfying
\begin{equation} \label{GrindEQ__1_12_}
u\in C\left(\mathbb R,\;H^{s} \right)\cap L^{\gamma(q)} \left(\mathbb R,\;H_{q,2}^{s} \right),
\end{equation}
for any B-admissible pair $\left(\gamma (q),q\right)$. Furthermore, there exist $u_{0}^{\pm } \in H^{s} $ such that
\begin{equation} \label{GrindEQ__1_13_}
{\mathop{\lim }\limits_{t\to \pm \infty }} \left\| u(t)-e^{it\Delta^{2} } u_{0}^{\pm } \right\| _{H^{s}} =0.
\end{equation}
\end{theorem}
\begin{remark}\label{rem 1.4.}
\textnormal{Since $0\le s_{c}\le \tilde{s}_{c}\le s$, we can see that $H^{s}\hookrightarrow \dot{H}^{s_{c}}\cap \dot{H}^{\tilde{s}_{c}}$ by using Lemmas \ref{lem 3.4.} and \ref{lem 3.6.}. This implies that the smallness condition on the initial data $u_{0}$ in Theorem \ref{thm 1.3.} is weaker than that of \cite{ARK222}, where the global well-posedness of \eqref{GrindEQ__1_1_} was established in the case that $\left\|u_{0}\right\|_{H^{s}}$ is sufficiently small.}
\end{remark}

This paper is organized as follows. In Section \ref{sec 2.}, we introduce some notation and recall some useful facts about Lorentz spaces. In Section \ref{sec 3.}, we give some remarks on Sobolev-Lorentz spaces and extend the fractional chain rule under Lorentz norms established by \cite{AT21}. In Section \ref{sec 4.}, we prove Theorems \ref{thm 1.1.} and \ref{thm 1.3.}.

\section{Preliminaries}\label{sec 2.}
Let us introduce some notation used in this paper. Throughout the paper, $\mathscr{F}$ denotes the Fourier transform, and the inverse Fourier transform is denoted by $\mathscr{F}^{-1}$. We also use the notation $\hat{f}$ instead of $\mathscr{F}f$. $C>0$ stands for a positive universal constant, which may be different at different places. $a\lesssim b$ means $a\le Cb$ for some constant $C>0$. $a\sim b$ expresses $a\lesssim b$ and $b\lesssim a$.
Given normed spaces $X$ and $Y$, $X\hookrightarrow Y$ means that $X$ is continuously embedded in $Y$.
For $p\in \left[1,\;\infty \right]$, $p'$ denotes the dual number of $p$, i.e. $1/p+1/p'=1$.
For $s\in \mathbb R$, we denote by $\left[s\right]$ the largest integer which is less than or equal to $s$ and by $\left\lceil s\right\rceil$ the minimal integer which is larger than or equal to $s$. For a multi-index $\alpha =\left(\alpha _{1} ,\;\alpha _{2},\;\ldots ,\;\alpha _{n} \right)$, denote
$$
D^{\alpha } =\partial _{x_{1} }^{\alpha _{1} } \cdots \partial _{x_{n} }^{\alpha _{n} } , \;\left|\alpha \right|=\left|\alpha _{1} \right|+\cdots +\;\left|\alpha _{n} \right|.
$$
For a function $f(z)$ defined for a complex variable $z$ and for a positive integer $k$, the $k$-th order derivative of $f(z)$ and its norm are defined by
$$
f^{(k)}(z):=\left(\frac{\partial ^{k} f}{\partial z^{k} } ,\; \frac{\partial ^{k} f}{\partial z^{k-1} \partial \bar{z}} ,\; {\dots},\;\frac{\partial ^{k} f}{\partial \bar{z}^{k} } \right),~|f^{\left(k\right)}(z)|:=\sum _{i=0}^{k}\left|\frac{\partial ^{k} f}{\partial z^{k-i} \partial \bar{z}^{i} } \right|,
$$
where
$$
\frac{\partial f}{\partial z}=\frac{1}{2} \left(\frac{\partial f}{\partial x} -i\frac{\partial f}{\partial y} \right),\; \frac{\partial f}{\partial \bar{z}}=\frac{1}{2} \left(\frac{\partial f}{\partial x} +i\frac{\partial f}{\partial y} \right).
$$
As in \cite{WHHG11}, for $s\in \mathbb R$ and $1<p<\infty $, we denote by $H_{p}^{s} (\mathbb R^{d} )$ and $\dot{H}_{p}^{s} (\mathbb R^{d} )$ the nonhomogeneous Sobolev space and homogeneous Sobolev space, respectively. The norms of these spaces are given as
$$
\left\| f\right\| _{H_{p}^{s} (\mathbb R^{d} )} =\left\| (I-\Delta)^{s/2} f\right\| _{L^{p} (\mathbb R^{d} )} , \;\left\| f\right\| _{\dot{H}_{p}^{s} (\mathbb R^{d} )} =\left\| (-\Delta)^{s/2} f\right\| _{L^{p} (\mathbb R^{d} )},
$$
where $(I-\Delta)^{s/2}f =\mathscr{F}^{-1} \left(1+|\xi|^{2} \right)^{s/2} \mathscr{F}f$ and $(-\Delta)^{s/2}f =\mathscr{F}^{-1} |\xi|^{s} \mathscr{F}f$. As usual, we abbreviate $H_{2}^{s} (\mathbb R^{d} )$ and $\dot{H}_{2}^{s} (\mathbb R^{d} )$ as $H^{s} (\mathbb R^{d} )$ and $\dot{H}^{s} (\mathbb R^{d} )$, respectively. For
$0<p,\; q\le \infty $, we denote by $L^{p,q} \left(\mathbb R^{d}
\right)$ the Lorentz space.
The quasi-norms of these spaces are given by
$$
\left\|f\right\|_{L^{p,q} (\mathbb R^{d})}=:\left(\int_{0}^{\infty}{\left(t^{\frac{1}{p}}f^{*}(t)\right)^{q}
\frac{dt}{t}}\right)^{\frac{1}{q}},~~\textnormal{when}~~0<q<\infty,
$$
$$
\left\|f\right\|_{L^{p,\infty} (\mathbb R^{d})}:=\sup_{t>0}t^{\frac{1}{p}}f^{*}(t),~~\textnormal{when}~~q=\infty,
$$
where $f^{*}(t)=\inf\left\{\tau:M^{n}\left(\left\{x:|f(x)|>\tau\right\}\right)\le t\right\}$, with $M^{n}$ being the Lebesgue measure in $\mathbb R^{d}$. Note that $L^{p,q} \left(\mathbb R^{d}
\right)$ is a quasi-Banach space for $0<p,\; q\le \infty $.  When $1<p<\infty$ and $1\le q \le \infty$,  $L^{p,q} \left(\mathbb R^{d}
\right)$ can be turned into a Banach space via an equivalent norm.  Note also that $L^{p,p} (\mathbb R^{d})=L^{p}
(\mathbb R^{d})$. See \cite{G14} for details. For $I\subset \mathbb R$ and $\gamma \in \left[1,\;\infty \right]$, we will use the space-time mixed space $L^{\gamma } \left(I,X\left(\mathbb R^{d}
\right)\right)$ whose norm is defined by
\[\left\| f\right\|_{L^{\gamma } \left(I,\;X(\mathbb R^{d})\right)}
=\left(\int _{I}\left\| f\right\| _{X(\mathbb R^{d})}^{\gamma } dt
\right)^{\frac{1}{\gamma } } ,\]
with a usual modification when $\gamma =\infty $, where $X(\mathbb R^{d})$
is a normed space on $\mathbb R^{d} $. If there is no confusion, $\mathbb R^{d} $ will be omitted in various function spaces.

Next, we recall some useful facts about Lorentz spaces.

\begin{lemma}[\cite{G14}]\label{lem 2.1.}
For $0<p<\infty$, $|x|^{-\frac{d}{p} } $ is in
$L^{p,\infty } (\mathbb R^{d})$ with the norm $v_{d}^{1/p}$, where $v_{d}$ is the measure of the unit ball of $\mathbb R^{d}$.
\end{lemma}

\begin{lemma}[\cite{G14}]\label{lem 2.2.}
For all $0<p,\;r<\infty $, $0<q\le
\infty $ we have
\[
\left\| \left|f\right|^{r} \right\| _{L^{p,q} } =\left\| f\right\| _{L^{pr,qr}
}^{r}.\]
\end{lemma}
\begin{lemma}[\cite{G14}]\label{lem 2.3.}
Suppose $0<p\le \infty $ and $0<q<r\le \infty $. Then we have
\[\left\| f\right\| _{L^{p,r} } \le C_{p,q,r} \left\| f\right\| _{L^{p,q} } .\]
\end{lemma}

\begin{lemma}[H\"{o}lder inequality in Lorentz spaces, \cite{O63}]\label{lem 2.4.}
Let $1<p,p_{1},p_{2}<\infty $ and $1\le q,q_{1} ,q_{2}\le \infty $ with
$$
\frac{1}{p}=\frac{1}{p_{1}} +\frac{1}{p_{2}} ,~\frac{1}{q}=\frac{1}{q_{1} } +\frac{1}{q_{2} }.
$$
Then we have
\[\left\| fg\right\| _{L^{p,q} } \lesssim \left\| f\right\|
_{L^{p_{1},q_{1} } } \left\| g\right\| _{L^{p_{2},q_{2} } } .\]
\end{lemma}

\begin{lemma}[Young inequality in Lorentz spaces, \cite{O63}]\label{lem 2.5.}
Let $1<p,p_{1},p_{2}<\infty $ and $1\le q,q_{1} ,q_{2}\le \infty $ with
$$
1+\frac{1}{p}=\frac{1}{p_{1}} +\frac{1}{p_{2}} ,~\frac{1}{q}=\frac{1}{q_{1} } +\frac{1}{q_{2} }.
$$
Then we have
\[\left\| f*g\right\| _{L^{p,q} } \lesssim \left\| f\right\| _{L^{p_{1},q_{1} } } \left\| g\right\| _{L^{p_{2},q_{2} } } .\]
\end{lemma}
\section{Sobolev-Lorentz spaces and Fractional chain rule}\label{sec 3.}
In this section, we give some remarks on Sobolev-Lorentz spaces and extend the fractional chain rule under Lorentz norms established by \cite{AT21}.
\subsection{Some remarks on Sobolev-Lorentz spaces}
In this subsection, we state the definition of Sobolev-Lorentz spaces and discuss their basic properties.
The Sobolev-Lorentz spaces are defined as follows. See also \cite{AT21, HYZ12} for example.
\begin{definition}\label{defn_3.1.}
\textnormal{Let $s\in \mathbb R$, $1<p<\infty$ and $1\le q \le \infty$. The homogeneous Sobolev-Lorentz space $\dot{H}^{s}_{p,q}(\mathbb R^{d})$ is defined as the set of functions satisfying $(-\Delta)^{s/2}f \in L^{p,q}(\mathbb R^{d})$, equipped with the norm
\[\left\|f\right\|_{\dot{H}^{s}_{p,q}(\mathbb R^{d})}:=\left\|(-\Delta)^{s/2}f\right\|_{L^{p,q}(\mathbb R^{d})}.\]
The nonhomogeneous Sobolev-Lorentz space ${H}^{s}_{p,q}(\mathbb R^{d})$ is defined as the set of functions satisfying $(I-\Delta)^{s/2}f \in L^{p,q}(\mathbb R^{d})$, equipped with the norm
\[\left\|f\right\|_{{H}^{s}_{p,q}(\mathbb R^{d})}:=\left\|(I-\Delta)^{s/2}f \right\|_{L^{p,q}(\mathbb R^{d})}.\]}
\end{definition}

Using the properties of Lorentz spaces, we immediately have
\begin{lemma}\label{lem 3.2.}
Let $s\in \mathbb R$, $1<p<\infty$ and $1\le q_{1}\le q_{2}\le \infty$. Then we have

$(a)$ $\dot{H}^{s}_{p,q_{1}} \hookrightarrow \dot{H}^{s}_{p,q_{2}},~H^{s}_{p,q_{1}} \hookrightarrow H^{s}_{p,q_{2}}$.

$(b)$ $\dot{H}^{s}_{p,p}=\dot{H}^{s}_{p}$,~$H^{s}_{p,p}=H^{s}_{p}$.
\end{lemma}

If $s\ge 0$, then we have the following equivalent norms of homogeneous and nonhomogeneous Sobolev-Lorentz spaces.
\begin{lemma}[\cite{AK213}]\label{lem 3.3.}
Let $s\ge 0$, $1<p<\infty $ and $1\le q\le \infty$. Then we have
\[ \left\|f\right\|_{\dot{H}_{p,q}^{s} }\sim \sum _{\left|\alpha \right|=[s]}\left\|D^{\alpha }f\right\| _{\dot{H}_{p,q}^{s-[s]}}.\]
\end{lemma}
\begin{lemma}\label{lem 3.4.}
Let $s\ge 0$, $1<p<\infty $ and $1\le q\le \infty$. Then we have $H^{s}_{p,q}=L^{p,q}\cap \dot{H}^{s}_{p,q}$ with
\[\left\|f\right\|_{H^{s}_{p,q}}\sim \left\|f\right\|_{L^{p,q}}+\left\|f\right\|_{\dot{H}^{s}_{p,q}}.\]
\end{lemma}
\begin{proof}
Using the well-known Mihlin multiplier theorem (see e.g. Proposition 1.12 of \cite{WHHG11}), we can see that $\rho_{1}(\xi):=\frac{(1+|\xi|^2)^{s/2}}{1+|\xi|^{s}}$ is a multiplier on $L^{p}$ for all $1<p<\infty$. We can deduce from the generalized Marcinkiewicz interpolation theorem (see Theorem 5.3.2 in \cite{BL76}) that $\rho_{1} (\xi)$ is a multiplier on $L^{p,q}$ for any $1<p<\infty$ and $1\le q \le \infty$. Thus we have
\begin{eqnarray}\begin{split}\nonumber
\left\|f\right\|_{H^{s}_{p,q}}&=\left\|\mathscr{F}^{-1}(1+|\xi|^2)^{\frac{s}{2}}\mathscr{F}f\right\|_{L^{p,q}}=\left\|\mathscr{F}^{-1}\rho_{1} (\xi)(1+|\xi|^{s})\mathscr{F}f
\right\|_{L^{p,q}}\\
&\lesssim \left\|\mathscr{F}^{-1}(1+|\xi|^{s})\mathscr{F}f\right\|_{L^{p,q}} \le \left\|f\right\|_{L^{p,q}}+\left\|f\right\|_{\dot{H}^{s}_{p,q}}.
\end{split}\end{eqnarray}
Conversely, we can also see that
\begin{equation}\label{GrindEQ__3_1_}
\rho_{2}^{s}(\xi):=\frac{1}{(1+|\xi|^2)^{\frac{s}{2}}},~\rho_{3}^{s}(\xi):=\frac{|\xi|^{s}}{(1+|\xi|^2)^{\frac{s}{2}}}
\end{equation}
are multipliers on $L^{p}$, and therefore on $L^{p,q}$ for any $1<p<\infty$ and $1\le q \le \infty$. Hence, we immediately get
\begin{eqnarray}\begin{split}\nonumber
\left\|f\right\|_{L^{p,q}}+\left\|f\right\|_{\dot{H}^{s}_{p,q}}&=
\left\|\mathscr{F}^{-1}\rho_{2}(\xi)(1+|\xi|^2)^{\frac{s}{2}}\mathscr{F}f\right\|_{L^{p,q}}+\left\|\mathscr{F}^{-1}\rho_{3}(\xi)(1+|\xi|^2)^{\frac{s}{2}}\mathscr{F}f\right\|_{L^{p,q}}\\
&\lesssim \left\|\mathscr{F}^{-1}(1+|\xi|^2)^{\frac{s}{2}}\mathscr{F}f\right\|_{L^{p,q}} \le \left\|f\right\|_{H^{s}_{p,q}},
\end{split}\end{eqnarray}
this completes the proof.
\end{proof}
We also have the following embeddings on the homogeneous and nonhomogeneous Sobolev-Lorentz spaces.
\begin{lemma}\label{lem 3.5.}
Let $-\infty < s_{2} \le s_{1} <\infty $ and $1<p_{1} \le p_{2} <\infty $ with $s_{1} -\frac{d}{p_{1} } =s_{2} -\frac{d}{p_{2} } $. Then for any $1\le q\le \infty$, there holds the embeddings:
$$
\dot{H}_{p_{1},q}^{s_{1} } \hookrightarrow \dot{H}_{p_{2},q}^{s_{2}},~H_{p_{1},q}^{s_{1} } \hookrightarrow H_{p_{2},q}^{s_{2}}
$$
\end{lemma}
\begin{proof}
Using Lemmas \ref{lem 2.1.}, Lemma \ref{lem 2.5.} and the fact
$$
\mathscr{F}^{-1}(|\xi|^{-\beta})=C_{d,\beta}|x|^{-(d-\beta)},~\textnormal{for}~\beta\in (0,d),
$$
we have
\begin{eqnarray}\begin{split}\nonumber
\left\|f\right\|_{\dot{H}_{p_{2},q}^{s_{2}} }&=\left\|\mathscr{F}^{-1}(|\xi|^{-(s_{1}-s_{2})})*(-\Delta)^{s_{1}/2}f\right\|_{L^{p_{2},q}}
=C \left\||x|^{-(d+s_{2}-s_{1})}*(-\Delta)^{s_{1}/2}f\right\|_{L^{p_{2},q}}\\
& \lesssim\left\||x|^{-(d+s_{2}-s_{1})}\right\|_{L^{\bar{p},\infty}}
\left\|(-\Delta)^{s_{1}/2}f\right\|_{L^{p_{1},q}}
\lesssim\left\|f\right\|_{\dot{H}_{p_{1},q}^{s_{1}}},
\end{split}\end{eqnarray}
and
\begin{eqnarray}\begin{split}\nonumber
\left\|f\right\|_{H_{p_{2},q}^{s_{2}} }&=\left\|\mathscr{F}^{-1}\rho_{3}^{s_{1}-s_{2}}(\xi)|\xi|^{-(s_{1}-s_{2})}(1+|\xi|^{2})^{\frac{s_{1}}{2}}\hat{f}\right\|_{L^{p_{2},q}}\\
&\lesssim \left\|\mathscr{F}^{-1}|\xi|^{-(s_{1}-s_{2})}(1+|\xi|^{2})^{\frac{s_{1}}{2}}\hat{f}\right\|_{L^{p_{2},q}}\\
&=C \left\||x|^{-(d+s_{2}-s_{1})}*(I-\Delta)^{\frac{s_{1}}{2}}f\right\|_{L^{p_{2},q}}\\
& \lesssim\left\||x|^{-(d+s_{2}-s_{1})}\right\|_{L^{\bar{p},\infty}}\left\|(I-\Delta)^{\frac{s_{1}}{2}}f\right\|_{L^{p_{1},q}}
\lesssim\left\|f\right\|_{H_{p_{1},q}^{s_{1}}},
\end{split}\end{eqnarray}
where $\frac{1}{\bar{p}}=1-\frac{1}{p_{1}}+\frac{1}{p_{2}}=\frac{d+s_{2}-s_{1}}{d}$ and $\rho_{3}^{s}(\xi)$ is given in \eqref{GrindEQ__3_1_}.
\end{proof}
\begin{lemma}\label{lem 3.6.}
Let $s\in \mathbb R$, $\varepsilon\ge 0$, $1<p<\infty $ and $1\le q\le \infty$. Then we have $H^{s+\varepsilon}_{p,q}\hookrightarrow H^{s}_{p,q}$.
\end{lemma}
\begin{proof}
The result follows directly from the fact that $\rho_{3}^{\varepsilon}(\xi)=(1+|\xi|^2)^{-\varepsilon/2}$ is a multiplier on $L^{p,q}$ for any $1<p<\infty$ and $1\le q \le \infty$.
\end{proof}

As an immediate consequence of Lemmas \ref{lem 3.5.} and \ref{lem 3.6.}, we have the following.
\begin{corollary}\label{cor 3.7.}
Let $-\infty < s_{2} \le s_{1} <\infty $ and $1<p_{1} \le p_{2} <\infty $ with $s_{1} -\frac{d}{p_{1} } \ge s_{2} -\frac{d}{p_{2} } $. Then for any $1\le q\le \infty$, there holds the embedding: $H_{p_{1},q}^{s_{1} } \hookrightarrow H_{p_{2},q}^{s_{2} }$.
\end{corollary}

\subsection{Fractional product rule and chain rule}
In this subsection, we recall the fractional product rule under Lorentz norms and extend the chain rule under Lorentz norms for the fractional Laplacian $(-\Delta)^{s/2}$ with $s\in (0,1]$ established by \cite{AT21} to any $s>0$.

First, we recall the fractional product rule under lorentz norms. See e.g. Theorem 6.1 of \cite{CN16}.
\begin{lemma}[Fractional product rule under Lorentz norms, \cite{CN16}]\label{lem 3.8.}
Let $s\ge 0$, $1<p,p_{1},p_{2},p_{3},p_{4} <\infty$ and $1\le q,q_{1},q_{2},q_{3},q_{4} \le \infty$. Assume that
\begin{equation}\nonumber
\frac{1}{p} =\frac{1}{p_{1} }+\frac{1}{p_{2}}=\frac{1}{p_{3} }+\frac{1}{p_{4}},~\frac{1}{q} =\frac{1}{q_{1} }+\frac{1}{q_{2}}=\frac{1}{q_{3} }+\frac{1}{q_{4}}.
\end{equation}
Then we have
\begin{equation}\nonumber
\left\|fg\right\| _{\dot{H}_{p,q}^{s} } \lesssim \left\|f\right\| _{\dot{H}_{p_1,q_1}^{s} } \left\| g\right\| _{L^{p_{2},q_{2}}} +\left\| f\right\| _{L^{p_{3},q_{3}}}
\left\| g\right\| _{\dot{H}_{p_4,q_4}^{s} }.
\end{equation}
\end{lemma}
Using Lemmas \ref{lem 3.8.} and \ref{lem 2.4.} and the induction, we immediately have the following result.

\begin{corollary}\label{cor 3.9.}
Let $s\ge 0$ and $k\in \mathbb N$. Let $1<p,p_{i_j} <\infty$ and $1\le q,q_{i_j} \le \infty$ for $1\le i,j\le k$. Assume that
\[\frac{1}{p} =\sum _{j=1}^{k}\frac{1}{p_{i_j}},~\frac{1}{q} =\sum _{j=1}^{k}\frac{1}{q_{i_j}} \]
for any $1\le i\le k$. Then we have
\begin{equation} \nonumber
\left\| \prod _{i=1}^{k}f_{i}  \right\| _{\dot{H}_{p,q}^{s} } \lesssim\sum _{i=1}^{k}(\left\| f_{i} \right\| _{\dot{H}_{p_{i_i},q_{i_i} }^{s} } \prod _{j\in I_{k}^{i} }\left\| f_{j} \right\| _{p_{i_j},q_{i_j} }  ) ,
\end{equation}
where $I_{k}^{i} =\left\{j\in \mathbb N:\;1\le j\le k,\;j\ne i\right\}$.
\end{corollary}

Next, we recall the chain rule for the fractional Laplacian $(-\Delta)^{s/2}$ with $s\in (0,1]$ established by \cite{AT21}.
\begin{lemma}[Fractional chain rule under Lorentz norms, \cite{AT21}]\label{lem 3.10.}
Suppose $F\in C^{1} (\mathbb C, \mathbb C)$ and $0< s\le1$. Then for $1<p,\;p_{1},\;p_{2} <\infty $ and $1\le q,\;q_{1},\; q_{2}< \infty $ satisfying
$$
\frac{1}{p} =\frac{1}{p_{1} } +\frac{1}{p_{2} },~\frac{1}{q} =\frac{1}{q_{1} } +\frac{1}{q_{2} },
$$
we have
\begin{equation} \nonumber
\left\|(-\Delta)^{s/2} F(u)\right\| _{L^{p,q}} \lesssim \left\| G'(u)\right\| _{L^{p_{1},q_{1}}} \left\|(-\Delta)^{s/2} u\right\| _{L^{p_{2},q_{2} }} .
\end{equation}
\end{lemma}

Lemma \ref{lem 3.10.} can be extended as follows.

\begin{lemma}\label{lem 3.11.}
Let $s>0$ and $F\in C^{\left\lceil s\right\rceil } \left(\mathbb C, \mathbb C\right)$. Then for $1<p,p_{1_k},p_{2_k},p_{3_k}<\infty$ and $1\le q,q_{1_k},q_{2_k},q_{3_k}<\infty$ satisfying
\begin{equation} \label{GrindEQ__3_2_}
\frac{1}{p}=\frac{1}{p_{1_k}}+\frac{1}{p_{2_k}}+\frac{k-1}{p_{3_k}},
~\frac{1}{q}=\frac{1}{q_{1_k}}+\frac{1}{q_{2_k}}+\frac{k-1}{q_{3_k}},~k=1,2,\ldots, \lceil s\rceil,
\end{equation}
we have
\begin{equation} \label{GrindEQ__3_3_}
\left\|(-\Delta)^{s/2} F(u)\right\| _{L^{p,q}} \lesssim \sum_{k=1}^{\lceil s\rceil}\left\| F^{(k)}(u)\right\| _{L^{p_{1_k},q_{1_k}}} \left\|(-\Delta)^{s/2}u\right\|_{L^{p_{2_k},q_{2_k}}}\left\|u\right\|_{L^{p_{3_k},q_{3_k}}}^{k-1}.
\end{equation}
\end{lemma}

As an immediate consequence of Lemma \ref{lem 3.11.}, we have the following useful nonlinear estimates in Sobolev-Lorentz spaces $\dot{H}_{p,q}^{s}$ with $s\ge 0$.
\begin{corollary}\label{cor 3.12.}
Let $s\ge 0$ and $\sigma > \max\{0, \lceil s\rceil -1\}$. Assume that $F\in C^{\left\lceil s\right\rceil } \left(\mathbb C, \mathbb C\right)$ satisfies
\begin{equation}\label{GrindEQ__3_4_}
|F^{(k)} (z)|\lesssim|z|^{\sigma +1-k} ,
\end{equation}
for any $0\le k\le \left\lceil s\right\rceil $ and $z \in \mathbb C$. Then for $1<p,p_{1},p_{2}<\infty$ and $1\le q,q_{1},q_{2}<\infty$ satisfying
\begin{equation} \label{GrindEQ__3_5_}
\frac{1}{p}=\frac{\sigma}{p_{1}}+\frac{1}{p_{2}},~\frac{1}{q}=\frac{\sigma}{q_{1}}+\frac{1}{q_{2}},
\end{equation}
we have
\begin{equation} \label{GrindEQ__3_6_}
\left\|F(u)\right\| _{\dot{H}_{p,q}^{s}} \lesssim \left\|u\right\| _{L^{p_{1},q_{1}}}^{\sigma} \left\|u\right\|_{\dot{H}_{p_{2},q_{2}}^{s}}.
\end{equation}
\end{corollary}

Corollary \ref{cor 3.12.} applies in particular to the model case $F(u)=|u|^{\sigma}u$ or $F(u)=|u|^{\sigma+1}$.
\begin{corollary}\label{cor 3.13.}
Let $F(u)=|u|^{\sigma}u$ or $F(u)=|u|^{\sigma+1}$. Assume $s\ge 0$ and $\sigma >0$. If $\sigma$ is not an even integer, assume further $\sigma >\lceil s\rceil -1$. Then for $1<p,p_{1},p_{2}<\infty$ and $1\le q,q_{1},q_{2}<\infty$ satisfying \eqref{GrindEQ__3_5_}, we have \eqref{GrindEQ__3_6_}.
\end{corollary}

In order to prove Lemma \ref{lem 3.11.}, we establish the following interpolation inequality in Sobolev-Lorentz spaces.
\begin{lemma}[Convexity H\"{o}lder inequality in Sobolev-Lorentz spaces]\label{lem 3.14.}
Let $1<p,\;p_{i} <\infty $, $1\le q,\;q_{i} <\infty $, $0\le \theta _{i} \le 1$, $s\ge 0$, $s_{i}\ge 0\; (i=1,\;\ldots ,\;N)$, $\sum _{i=1}^{N}\theta _{i} =1$, $s=\sum _{i=1}^{N}\theta _{i}  s_{i} $, $1/p=\sum _{i=1}^{N}{\theta _{i}/p_{i}}$ and $1/q=\sum _{i=1}^{N}{\theta _{i}/q_{i}}$.
Then we have $\bigcap _{i=1}^{N}\dot{H}_{p_{i},q_{i} }^{s_{i} }  \subset \dot{H}_{p,q}^{s} $ and for any $f\in \bigcap _{i=1}^{N}\dot{H}_{p_{i},q_{i} }^{s_{i} }  $,
\begin{equation} \nonumber
\left\|f\right\| _{\dot{H}_{p,q}^{s} } \le \prod _{i=1}^{N}\left\| f\right\| _{\dot{H}_{p_{i},q_{i} }^{s_{i} } }^{\theta _{i} }  .
\end{equation}
\end{lemma}
\begin{proof}
Let us introduce $\eta\in C_{0}^{\infty}(\mathbb R^{d})$, nonnegative function supported in $\{1/2<|\xi|<2\}$ and satisfying
$$
\sum_{j=-\infty}^{\infty}\eta(2^{j}\xi)\equiv 1,~\xi\in \mathbb R^{d}\setminus\{0\}.
$$
Define Fourier multiplier operators
$$
Q_{j}f=\mathscr{F}^{-1}\left(\eta(2^{-j}\xi)\hat{f}\right),~f\in S'(\mathbb R^{d}),~j\in \mathbb Z.
$$
Considering the square-function operator
$$
S(f)(x)=\left(\sum_{j\in \mathbb Z}{|Q_{j}(f)(x)|^{2}}\right)^{1/2},~x\in \mathbb R^{d},
$$
it was proved in Lemma 2.5 of \cite{AT21} that
\begin{equation} \label{GrindEQ__3_7_}
\left\|f\right\|_{L^{p,q}}\sim \left\|S(f)\right\|_{L^{p,q}},~ \textrm{for}~ 1<p<\infty,~ 1\le q<\infty.
\end{equation}
Using \eqref{GrindEQ__3_7_} and the fact that (see page 5418 in \cite{AT21})
\begin{equation} \nonumber
\left(\sum_{j\in \mathbb Z}{|Q_{j}((-\Delta)^{s/2}f)|^{2}}\right)^{1/2}
\sim \left(\sum_{j\in \mathbb Z}{2^{2js}|\tilde{Q}_{j}(f)|^{2}}\right)^{1/2},
\end{equation}
where
$$
\tilde{Q}_{j}f=\mathscr{F}^{-1}\left(\tilde{\eta}(2^{-j}\xi)\hat{f}\right),~f\in S'(\mathbb R^{d}),~j\in \mathbb Z.
$$
for $\tilde{\eta}\in C_{0}^{\infty}(\{1/2<|\xi|<2\})$ satisfying $\tilde{\eta}\cdot\eta\equiv \eta$, we have
\begin{eqnarray}\begin{split} \label{GrindEQ__3_8_}
\left\|(-\Delta)^{s/2}f\right\|_{L^{p,q}}&\sim \left\|S((-\Delta)^{s/2}f)\right\|_{L^{p,q}}
&\sim \left\|\left(\sum_{j\in \mathbb Z}{2^{2js}|\tilde{Q}_{j}(f)|^{2}}\right)^{1/2}\right\|_{L^{p,q}}.
\end{split}\end{eqnarray}
Using \eqref{GrindEQ__3_8_} and H\"{o}lder inequality, we have
\footnote[2] {For $1\le p,q,r<\infty$, $\left\|\{a_{j}(x)\}\right\|_{L^{p,q}(l^{r})}:=\left\|\left\|\{a_{j}(x)\}\right\|_{l^{r}}\right\|_{L^{p,q}}=\left\|(\sum_{j}|a_{j}(x)|^{r})^{1/r}\right\|_{L^{p,q}}$.}
\begin{eqnarray}\begin{split} \nonumber
\left\|(-\Delta)^{s/2}f\right\|_{L^{p,q}}&\lesssim \left\|\{2^{js}\tilde{Q}_{j}(f)\}\right\|_{L^{p,q}(l^{2})}
=\left\|\left\{\prod _{i=1}^{N}(2^{js_{i}}\tilde{Q}_{j}(f))^{\theta_{i}}\right\}\right\|_{L^{p,q}(l^{2})}\\
&\lesssim \left\|\prod _{i=1}^{N} \left\|\{2^{js_{i}}\tilde{Q}_{j}(f)\}\right\|_{l^{2}}^{\theta_{i}}\right\|_{L^{p,q}}
\lesssim \prod _{i=1}^{N}{\left\|\{2^{js_{i}}\tilde{Q}_{j}(f)\}\right\|_{L^{p_{i},q_{i}}(l^{2})}^{\theta_{i}}}\\
&\lesssim \prod _{i=1}^{N}\left\|(-\Delta)^{s_{i}/2}f\right\| _{L^{p_{i},q_{i} } }^{\theta _{i}},
\end{split}\end{eqnarray}
this completes the proof.
\end{proof}

\begin{proof}[{\bf Proof of Lemma \ref{lem 3.11.}}]
The proof in the case $s\le 1$ can be found in Lemma 2.4 of \cite{AT21}. So it suffices to consider the case $s>1$.
By Lemma \ref{lem 3.3.}, we have
\begin{equation} \label{GrindEQ__3_9_}
\left\| F(u)\right\|_{\dot{H}_{p,q}^{s} } \lesssim\sum_{|\alpha |=[s]}\left\| D^{\alpha } F(u)\right\| _{\dot{H}_{p,q}^{v}},
\end{equation}
where $v=s-[s]$. Without loss of generality and for simplicity, we assume that $F$ is a function of a real variable. It follows from the Leibniz rule of derivatives that
\[D^{\alpha } F(u)=\sum _{k=1}^{|\alpha |}\sum _{\Lambda _{\alpha }^{k} }C_{\alpha ,\;k} F^{(k)} (u)\prod _{i=1}^{k}D^{\alpha _{i} } u ,\]
where $\Lambda _{\alpha }^{k} =\left(\alpha _{1} +\cdots+\alpha _{i}+\cdots+\alpha _{k} =\alpha ,\;|\alpha _{i}|\ge 1\right)$.
Hence we have
\begin{eqnarray}\begin{split} \label{GrindEQ__3_10_}
\left\| F(u)\right\|_{\dot{H}_{p,q}^{s} } &\lesssim \sum_{|\alpha |=[s]}\sum _{k=1}^{|\alpha |}\sum _{\Lambda _{\alpha }^{k} } \left\|  F^{(k)} (u)\prod _{i=1}^{k}D^{\alpha _{i} } u  \right\| _{\dot{H}_{p,q}^{v} }.
\end{split}\end{eqnarray}

Let us estimate
\begin{equation} \label{GrindEQ__3_11_}
I_{k}\equiv\left\|  F^{(k)} (u)\prod _{i=1}^{k}D^{\alpha _{i} } u  \right\| _{\dot{H}_{p,q}^{v} },
\end{equation}
where $1\le k\le [s]$, $|\alpha _{1}|+\cdots+|\alpha _{i}|+\cdots+|\alpha _{k}|=[s]$ and $|\alpha _{i}|\ge 1$.

We divide the study in two cases: $s\in \mathbb N$ and $s\notin \mathbb N$.

{\bf Case 1.} We consider the case $s\in \mathbb N$, i.e. $v=0$.
For $i=1,\ldots k$, we define
\begin{equation} \label{GrindEQ__3_12_}
\frac{1}{\alpha_{k_i}}:=\frac{|\alpha_{i}|}{s}\frac{1}{p_{2_k}}+\left(1-\frac{|\alpha_{i}|}{s}\right)\frac{1}{p_{3_k}},~
\frac{1}{\beta_{k_i}}:=\frac{|\alpha_{i}|}{s}\frac{1}{q_{2_k}}+\left(1-\frac{|\alpha_{i}|}{s}\right)\frac{1}{q_{3_k}}.
\end{equation}
Using \eqref{GrindEQ__3_11_}, lemmas \ref{lem 2.4.}, \ref{lem 3.3.} and \ref{lem 3.14.}, we have
\begin{eqnarray}\begin{split}\label{GrindEQ__3_13_}
I_{k}& \le \left\| F^{(k)} (u)\right\| _{L^{p_{1_k},q_{1_k}}} \prod_{i=1}^{k}{\left\|D^{\alpha _{i} } u\right\|_{L^{\alpha_{k_i},\beta_{k_i}}}}\\
&\lesssim \left\| F^{(k)} (u)\right\| _{L^{p_{1_k},q_{1_k}}}
\prod_{i=1}^{k} (\left\|u\right\| _{\dot{H}_{p_{2_k},q_{2_k}}^{s}}^{\frac{|\alpha_{i}|}{s}}
\left\|u\right\| _{L^{p_{3_k},q_{3_k}}}^{1-\frac{|\alpha_{i}|}{s}})\\
& \lesssim \left\| F^{(k)} (u)\right\| _{L^{p_{1_k},q_{1_k}}}\left\|u\right\| _{\dot{H}_{p_{2_k},q_{2_k}}^{s}} \left\|u\right\| _{L^{p_{3_k},q_{3_k}}}^{k-1}.
\end{split}\end{eqnarray}
In view of \eqref{GrindEQ__3_10_}, \eqref{GrindEQ__3_11_} and \eqref{GrindEQ__3_13_}, we get the desired result.

{\bf Case 2.} We consider the case $s\notin \mathbb N$.
For $1\le k\le [s]$, putting
\begin{equation} \label{GrindEQ__3_14_}
\frac{1}{a_{k}} =\frac{1}{p_{2_k}}+\frac{k-1}{p_{3_k}},~\frac{1}{b_{k}} =\frac{1}{q_{2_k}}+\frac{k-1}{q_{3_k}},
\end{equation}
\begin{equation} \label{GrindEQ__3_15_}
\frac{1}{c_{k}} =\frac{1}{p_{1_{k+1}}}+\frac{v}{s}\frac{1}{p_{2_{k+1}}}+\frac{[s]}{s}\frac{1}{p_{3_{k+1}}},~
\frac{1}{d_{k}} =\frac{1}{q_{1_{k+1}}}+\frac{v}{s}\frac{1}{q_{2_{k+1}}}+\frac{[s]}{s}\frac{1}{q_{3_{k+1}}},
\end{equation}
and
\begin{equation} \label{GrindEQ__3_16_}
\frac{1}{e_{k}} =\frac{[s]}{s}\frac{1}{p_{2_{k+1}}}+\left(k-\frac{[s]}{s}\right)\frac{1}{p_{3_{k+1}}},~
\frac{1}{f_{k}} =\frac{[s]}{s}\frac{1}{q_{2_{k+1}}}+\left(k-\frac{[s]}{s}\right)\frac{1}{q_{3_{k+1}}},
\end{equation}
we deduce from \eqref{GrindEQ__3_2_} that
\begin{equation} \label{GrindEQ__3_17_}
\frac{1}{p}=\frac{1}{p_{1_{k+1}}}+\frac{1}{a_{k}}=\frac{1}{c_k}+\frac{1}{e_k},~\frac{1}{q}=\frac{1}{q_{1_{k+1}}}+\frac{1}{b_{k}}=\frac{1}{d_k}+\frac{1}{f_k}.
\end{equation}
Using \eqref{GrindEQ__3_17_} and Lemma \ref{lem 3.8.} (fractional product rule), we have
\begin{eqnarray}\begin{split} \label{GrindEQ__3_18_}
I_{k}&\lesssim \left\| F^{(k)} (u)\right\| _{L^{p_{1_k},q_{1_k}}} \left\| \prod _{i=1}^{k}D^{\alpha _{i} } u \right\| _{\dot{H}_{a_k,b_k}^{v}}
+\left\| F^{(k)} (u)\right\| _{\dot{H}_{c_k,d_k}^{v}}\left\| \prod _{i=1}^{k}D^{\alpha _{i} } u \right\| _{L^{e_k,f_k}}\\
&~~:=I_{k_1}+I_{k_2}.
\end{split}\end{eqnarray}

First, we estimate $I_{k_1}$. If $k=1$, we can see that $\left|\alpha _{1} \right|=[s]$, $a_{k} =p_{2_k}$ and $b_{k} =q_{2_k}$. Hence, we immediately have
\[\left\| D^{\alpha _{1}} u\right\| _{\dot{H}_{a_{k},b_{k} }^{v} } \lesssim\left\| u\right\| _{\dot{H}_{p_{2_k},q_{2_k}}^{s} } .\]
We consider the case $k>1$. For $1\le i\le k$, putting

\begin{equation} \label{GrindEQ__3_19_}
\frac{1}{a_{k_i}}:=\frac{|\alpha_{i}|}{s}\frac{1}{p_{2_k}}+\left(1-\frac{|\alpha_{i}|}{s}\right)\frac{1}{p_{3_k}},~
\frac{1}{b_{k_i}}:=\frac{|\alpha_{i}|}{s}\frac{1}{q_{2_k}}+\left(1-\frac{|\alpha_{i}|}{s}\right)\frac{1}{q_{3_k}},
\end{equation}
\begin{equation} \label{GrindEQ__3_20_}
\frac{1}{\tilde{a}_{k_i}}:=\frac{|\alpha_{i}|+v}{s}\frac{1}{p_{2_k}}+\left(1-\frac{|\alpha_{i}|+v}{s}\right)\frac{1}{p_{3_k}},~
\frac{1}{\tilde{b}_{k_i}}:=\frac{|\alpha_{i}|+v}{s}\frac{1}{q_{2_k}}+\left(1-\frac{|\alpha_{i}|+v}{s}\right)\frac{1}{q_{3_k}},
\end{equation}
we can see that
\begin{equation} \label{GrindEQ__3_21_}
\frac{1}{a_{k} } =\sum _{j\in I_{k}^{i} }\frac{1}{a_{k_j} }  +\frac{1}{\tilde{a}_{k_i} } ,~\frac{1}{b_{k} } =\sum _{j\in I_{k}^{i} }\frac{1}{b_{k_j} }  +\frac{1}{\tilde{b}_{k_i} },
\end{equation}
where $I_{k}^{i} =\left\{j\in \mathbb N:\;1\le j\le k,\;j\ne i\right\}$.
By using Corollary \ref{cor 3.9.}, Lemma \ref{lem 3.14.} and \eqref{GrindEQ__3_21_}, we have
\begin{eqnarray}\begin{split} \label{GrindEQ__3_22_}
\left\| \prod _{i=1}^{k}D^{\alpha _{i} } u \right\| _{\dot{H}_{a_k,b_k}^{v}}
&\lesssim \sum _{i=1}^{k}(\left\| D^{\alpha _{i} } u\right\| _{\dot{H}_{\tilde{a}_{k_i},\tilde{b}_{k_i} }^{v} } \prod _{j\in I_{k}^{i} }\left\| D^{\alpha _{j} } u \right\| _{L^{a_{k_j},b_{k_j}}})\\
&\lesssim\sum _{i=1}^{k}(\left\| u\right\| _{\dot{H}_{\tilde{a}_{k_i},\tilde{b}_{k_i} }^{|\alpha _{i}|+v} } \prod _{j\in I_{k}^{i} }\left\| u\right\| _{\dot{H}_{a_{k_j},b_{k_j} }^{|\alpha _{j}|} })\\
&\lesssim \sum _{i=1}^{k} \left(\left\| u\right\| _{\dot{H}_{p_{2_k},q_{2_k}}^{s} }^{\frac{|\alpha _{i}|+v}{s}}\left\|u\right\|_{L^{p_{3_k},q_{3_k}}}^{1-\frac{|\alpha _{i}|+v}{s}}
\prod _{j\in I_{k}^{i} }(\left\| u\right\| _{\dot{H}_{p_{2_k},q_{2_k}}^{s} }^{\frac{|\alpha _{j}|}{s}}\left\|u\right\|_{L^{p_{3_k},q_{3_k}}}^{1-\frac{|\alpha _{j}|}{s}})\right)\\
&=\left\| u\right\| _{\dot{H}_{p_{2_k},q_{2_k}}^{s} }\left\|u\right\|_{L^{p_{3_k},q_{3_k}}}^{k-1}.
\end{split}\end{eqnarray}
Hence, for any $1\le k\le [s]$, we have
\begin{equation} \label{GrindEQ__3_23_}
I_{k_1} \lesssim \left\| F^{(k)}(u)\right\| _{L^{p_{1_k},q_{1_k}}} \left\| u\right\| _{\dot{H}_{p_{2_k},q_{2_k}}^{s} }\left\|u\right\|_{L^{p_{3_k},q_{3_k}}}^{k-1}.
\end{equation}

Next, we estimate $I_{k_2}$.
Using \eqref{GrindEQ__3_15_}, Lemmas \ref{lem 3.10.} and \ref{lem 3.14.}, we have
\begin{eqnarray}\begin{split} \label{GrindEQ__3_24_}
\left\| F^{(k)} (u)\right\| _{\dot{H}_{c_k,d_k}^{v}}&\lesssim \left\| F^{(k+1)} (u)\right\| _{L^{p_{1_k},q_{1_k}}}\left\| u\right\| _{\dot{H}_{g_k,h_k}^{v}}\\
&\lesssim \left\| F^{(k+1)} (u)\right\| _{L^{p_{1_k},q_{1_k}}} \left\| u\right\| _{\dot{H}_{p_{2_k},q_{2_k}}^{s}}^{\frac{v}{s}}\left\|u\right\|_{L^{p_{3_k},q_{3_k}}}^{\frac{[s]}{s}},
\end{split}\end{eqnarray}
where
$$
\frac{1}{g_k}:=\frac{v}{s}\frac{1}{p_{2_{k+1}}}+\frac{[s]}{s}\frac{1}{p_{3_{k+1}}},~\frac{1}{h_k}:=\frac{v}{s}\frac{1}{q_{2_{k+1}}}+\frac{[s]}{s}\frac{1}{q_{3_{k+1}}}.
$$
It also follows from \eqref{GrindEQ__3_16_}, \eqref{GrindEQ__3_19_} and Lemma \ref{lem 3.14.} that
\begin{eqnarray}\begin{split} \label{GrindEQ__3_25_}
\left\| \prod _{i=1}^{k}D^{\alpha _{i} } u \right\| _{L^{e_k,f_k}}
&\lesssim \prod _{i=1}^{k} \left\| D^{\alpha _{i} } u \right\| _{L^{a_{k_i},b_{k_i}}}
\lesssim \prod _{i=1}^{k}(\left\| u\right\| _{\dot{H}_{p_{2_k},q_{2_k}}^{s} }^{\frac{|\alpha _{i}|}{s}}\left\|u\right\|_{L^{p_{3_k},q_{3_k}}}^{1-\frac{|\alpha _{i}|}{s}})\\
&=\left\| u\right\| _{\dot{H}_{p_{2_k},q_{2_k}}^{s} }^{\frac{[s]}{s}}\left\|u\right\|_{L^{p_{3_k},q_{3_k}}}^{k-\frac{[s]}{s}}.
\end{split}\end{eqnarray}
In view of \eqref{GrindEQ__3_24_} and \eqref{GrindEQ__3_25_}, we have
\begin{equation} \label{GrindEQ__3_26_}
I_{k_2} \lesssim \left\| F^{(k+1)} (u)\right\| _{L^{p_{1_k},q_{1_k}}} \left\| u\right\| _{\dot{H}_{p_{2_k},q_{2_k}}^{s}}\left\|u\right\|_{L^{p_{3_k},q_{3_k}}}^{k},
\end{equation}
for any $1\le k\le [s]$. In view of \eqref{GrindEQ__3_9_}, \eqref{GrindEQ__3_10_}, \eqref{GrindEQ__3_18_}, \eqref{GrindEQ__3_23_} and \eqref{GrindEQ__3_26_}, we can get \eqref{GrindEQ__3_3_}. This completes the proof.
\end{proof}


\section{Well-posedness for the IBNLS equation}\label{sec 4.}

In this section, we establish the local and global well-posedness for the Cauchy problem of the IBNLS equation \eqref{GrindEQ__1_1_}, i.e. we prove Theorems \ref{thm 1.1.} and \ref{thm 1.3.}.

First of all, we recall the Strichartz estimates in Lorentz spaces, which follows from Theorem 10.1 of \cite{KT98} with the notation of \cite{KT98} we take $\sigma=\frac{d}{4}$, $H=B_{0}=L^{2}$, $B_{1}=L^{1}$, $\theta=1-\frac{4}{p}$, $B_{\theta}=L^{p',2}$.

\begin{lemma}[Strichartz estimates in Lorentz spaces]\label{lem 4.1.}
Let $S(t)=e^{it\Delta^{2} } $. Then for any admissible pairs $(\gamma(p),p)$ and $(\gamma(r),r)$, we have
\begin{equation} \label{GrindEQ__4_1_}
\left\| S(t)\phi \right\| _{L^{\gamma(p)}(\mathbb R,L^{p,2})} \lesssim\left\| \phi \right\| _{L^{2} } ,
\end{equation}
\begin{equation} \label{GrindEQ__4_2_}
\left\|\int_{0}^{t}S(t-\tau)f(\tau)d\tau\right\|_{L^{\gamma(p)}(\mathbb R,L^{p,2})}\lesssim\left\|f\right\|_{L^{\gamma(r)'}(\mathbb R,L^{r',2})}.
\end{equation}
\end{lemma}

Next, we recall the useful fact concerning to the term $|x|^{-b}$ with $b>0$.
\begin{remark}[\cite{G17}]\label{rem 4.2.}
\textnormal{Let $b>0$, $s\ge 0$ and $b+s<d$. Then we have $(-\Delta)^{s/2}(|x|^{-b})=C_{d,b}|x|^{-b-s}$.}
\end{remark}
\subsection{Local well-posedness}
In this section, we prove Theorem \ref{thm 1.1.}.
To this end, we get the following nonlinear estimates.
\begin{lemma}\label{lem 4.3.}
Let $0\le s< \min\{2+\frac{d}{2},\;d\}$, $0<b<\min\{4,\;2+\frac{d}{2}-s,\;d-s\}$ and $0<\sigma\le \sigma_{c}(s)$ with $\sigma<\infty$. If $\sigma$ is not an even integer, assume that $\sigma>\left\lceil s\right\rceil -1$. Then for any interval $I(\subset \mathbb R)$, there exist $B$-admissible pairs $(\gamma(p),p)$ and $(\gamma(r),r)$ such that
\begin{equation} \label{GrindEQ__4_3_}
\left\||u|^{\sigma}u\right\|_{L^{\gamma(p)'}(I,\; \dot{H}_{p',2}^{s})}
\lesssim |I|^{\theta}\left\|u\right\|^{\sigma+1}_{L^{\gamma(r)}(I,\; H_{r,2}^{s})},
\end{equation}
\begin{equation} \label{GrindEQ__4_4_}
\left\||u|^{\sigma}v\right\|_{L^{\gamma(p)'}(I,\; L^{p',2})}
\lesssim |I|^{\theta}\left\|u\right\|^{\sigma}_{L^{\gamma(r)}(I,\; H_{r,2}^{s})}\left\|u\right\|_{L^{\gamma(r)}(I,\; L^{r,2})},
\end{equation}
where $\theta>0$ in the subcritical case $\sigma<\sigma_{c}(s)$, and $\theta=0$ in the critical case $\sigma=\sigma_{c}(s)$.
\end{lemma}
\begin{proof}
We divide the proof in two cases: $0\le s<\frac{d}{2}$ and $s\ge \frac{d}{2}$.

{\bf Case 1.} We consider the case $0\le s<\frac{d}{2}$. Let $B$-admissible pairs $(\gamma(p),p)$ and $(\gamma(r),r)$ satisfy
\begin{equation} \label{GrindEQ__4_5_}
\frac{1}{p'} =\sigma \left(\frac{1}{r} -\frac{s}{d} \right)+\frac{1}{r}+\frac{b}{d}, ~\frac{1}{r} -\frac{s}{d} >0.
\end{equation}
Using Lemma \ref{lem 3.8.}, Lemma \ref{lem 2.1.} and Remark \ref{rem 4.2.}, we have
\begin{eqnarray}\begin{split} \label{GrindEQ__4_6_}
\left\| |x|^{-b}|u|^{\sigma}u\right\| _{\dot{H}_{p',2}^{s} } &\lesssim \left\| |x|^{-b}\right\| _{L^{p_{1},\infty}}\left\| |u|^{\sigma}u\right\| _{\dot{H}_{p_{2},2}^{s}}
+\left\| |x|^{-b}\right\| _{\dot{H}_{p_{3},\infty}^{s}}\left\| |u|^{\sigma}u\right\| _{L^{p_{4},2}}\\
&\lesssim \left\| |u|^{\sigma}u\right\| _{\dot{H}_{p_{2},2}^{s}}+\left\| |u|^{\sigma}u\right\| _{L^{p_{4},2}},
\end{split}\end{eqnarray}
where
\begin{equation}\label{GrindEQ__4_7_}
\frac{1}{p_{1}}:=\frac{b}{d},~\frac{1}{p_{2}}:=\frac{1}{p'}-\frac{b}{d},~\frac{1}{p_{3}}:=\frac{b+s}{d},~\frac{1}{p_{4}}:=\frac{1}{p'}-\frac{b+s}{d}.
\end{equation}
Putting $\frac{1}{\alpha}:=\frac{1}{r}-\frac{s}{d}$, it follows from Lemma \ref{lem 3.5.} that $\dot{H}_{r,2}^{s}\hookrightarrow L^{\alpha, 2}$. Hence, using \eqref{GrindEQ__4_5_}--\eqref{GrindEQ__4_7_}, Lemmas \ref{lem 2.3.}, \ref{lem 2.4.} and Corollary \ref{cor 3.13.}, we immediately have
\begin{eqnarray}\begin{split} \label{GrindEQ__4_8_}
\left\| |x|^{-b}|u|^{\sigma}u\right\| _{\dot{H}_{p',2}^{s} } &\lesssim \left\| |u|^{\sigma}u\right\| _{\dot{H}_{p_{2},2}^{s}}+\left\| |u|^{\sigma}u\right\| _{L^{p_{4},2}}
\lesssim \left\| u\right\| _{L^{\alpha,2(\sigma+1)}}^{\sigma}\left\| u\right\| _{\dot{H}_{r,2(\sigma+1)}^{s}}+\left\| u\right\| _{L^{\alpha,2(\sigma+1)}}^{\sigma+1}\\
&\lesssim \left\| u\right\| _{L^{\alpha,2}}^{\sigma}\left\| u\right\| _{\dot{H}_{r,2}^{s}}+\left\| u\right\| _{L^{\alpha,2}}^{\sigma+1}
\lesssim \left\| u\right\| _{\dot{H}_{r,2}^{s} }^{\sigma +1}.
\end{split}\end{eqnarray}
Similarly, we also have
\begin{eqnarray}\begin{split} \label{GrindEQ__4_9_}
\left\| |x|^{-b}|u|^{\sigma}v\right\| _{L^{p',2}} &\lesssim  \left\| |x|^{-b}\right\| _{L^{p_{1},\infty}}\left\| |u|^{\sigma}v\right\| _{L^{p_{2},2}}
\lesssim \left\| u\right\| _{L^{\alpha,2}}^{\sigma}\left\| v\right\| _{L^{r,2}}
\lesssim \left\| u\right\| _{\dot{H}_{r,2}^{s} }^{\sigma }\left\| v\right\| _{L^{r,2}}.
\end{split}\end{eqnarray}
Meanwhile, it follows from \eqref{GrindEQ__4_5_} that
\begin{equation}\label{GrindEQ__4_10_}
\theta:=\frac{1}{\gamma(p)'}-\frac{\sigma+1}{\gamma(r)}=\frac{4-b}{4}-\frac{(d-2s)\sigma}{8}
\end{equation}
Using \eqref{GrindEQ__4_8_}, \eqref{GrindEQ__4_9_}, \eqref{GrindEQ__4_10_}, H\"{o}lder inequality and the fact $\sigma\le \frac{8-2b}{d-2s}$, we immediately get \eqref{GrindEQ__4_3_} and \eqref{GrindEQ__4_4_}. So it suffices to prove that there exist $B$-admissible pairs $(\gamma(p),p)$ and $(\gamma(r),r)$ satisfying \eqref{GrindEQ__4_5_}.
Let
$$
\max\left\{\frac{d-4}{2d},\;\frac{s}{d}\right\}<\frac{1}{r}<\frac{1}{2}.
$$
Then it follows from \eqref{GrindEQ__4_5_} that
\begin{equation}\label{GrindEQ__4_11_}
\max\left\{\frac{s}{d}+\frac{b}{d},\frac{(d-4)(\sigma+1)}{2d}-\frac{\sigma s}{d}+\frac{b}{d}\right\}<
\frac{1}{p'}<\frac{\sigma+1}{2}-\frac{\sigma s}{d}+\frac{b}{d}.
\end{equation}
We can easily check that there exists a $B$-admissible pair $(\gamma(p),p)$ satisfying \eqref{GrindEQ__4_11_} by using the fact $b<\min\{d-s,2+\frac{d}{2}-s\}$.

{\bf Case 2.} We consider the case $s\ge \frac{d}{2}$.
Let $B$-admissible pairs $(\gamma(p),p)$ and $(\gamma(r),r)$ satisfy
\begin{equation} \label{GrindEQ__4_12_}
\left\{\begin{array}{l}
{\frac{1}{p'}=\frac{\sigma}{\alpha}+\frac{1}{r}+\frac{b}{d}=\frac{\sigma+1}{\beta}+\frac{b+s}{d},}\\
{\theta:=\frac{1}{\gamma(p)'}-\frac{\sigma+1}{\gamma(r)}}>0,\\
\end{array}\right.
\end{equation}
for some $r<\alpha,\beta<\infty$. Noticing $2\le r<\alpha<\infty$, and using Corollary \ref{cor 3.7.}, we have $H^{s}_{r,2}\hookrightarrow L^{\alpha,2}$. Similarly, we also have $H^{s}_{r,2}\hookrightarrow L^{\beta,2}$. Hence, using Lemmas \ref{lem 2.1.}, \ref{lem 2.3.}, \ref{lem 2.4.}, \ref{lem 3.8.}, Corollary \ref{cor 3.13.}, Remark \ref{rem 4.2.} and the firs equation in \eqref{GrindEQ__4_12_}, we have
\begin{eqnarray}\begin{split} \label{GrindEQ__4_13_}
\left\| |x|^{-b}|u|^{\sigma}u\right\| _{\dot{H}_{p',2}^{s} } &\lesssim \left\| |x|^{-b}\right\| _{L^{p_{1},\infty}}\left\| |u|^{\sigma}u\right\| _{\dot{H}_{p_{2},2}^{s}}
+\left\| |x|^{-b}\right\| _{\dot{H}_{p_{3},\infty}^{s}}\left\| |u|^{\sigma}u\right\| _{L^{p_{4},2}}\\
&\lesssim \left\| |u|^{\sigma}u\right\| _{\dot{H}_{p_{2},2}^{s}}+\left\| |u|^{\sigma}u\right\| _{L^{p_{4},2}}\\
&\lesssim \left\| u\right\| _{L^{\alpha,2}}^{\sigma}\left\| u\right\| _{\dot{H}_{r,2}^{s}}+\left\| u\right\| _{L^{\beta,2}}^{\sigma+1}
\lesssim \left\| u\right\| _{\dot{H}_{r,2}^{s} }^{\sigma +1},
\end{split}\end{eqnarray}
and
\begin{eqnarray}\begin{split} \label{GrindEQ__4_14_}
\left\| |x|^{-b}|u|^{\sigma}v\right\| _{L^{p',2}} &\lesssim  \left\| |x|^{-b}\right\| _{L^{p_{1},\infty}}\left\| |u|^{\sigma}v\right\| _{L^{p_{2},2}}
\lesssim \left\| u\right\| _{L^{\alpha,2}}^{\sigma}\left\| v\right\| _{L^{r,2}}
\lesssim \left\| u\right\| _{\dot{H}_{r,2}^{s} }^{\sigma }\left\| v\right\| _{L^{r,2}},
\end{split}\end{eqnarray}
where $p_{i}~(i=\overline{1,4})$  are as in \eqref{GrindEQ__4_7_}.
Using the second equation in \eqref{GrindEQ__4_12_}, \eqref{GrindEQ__4_13_}, \eqref{GrindEQ__4_14_} and H\"{o}lder inequality, we immediately get \eqref{GrindEQ__4_3_} and \eqref{GrindEQ__4_4_}. It remains to prove that there exist $B$-admissible pairs $(\gamma(p),p)$ and $(\gamma(r),r)$ satisfying \eqref{GrindEQ__4_12_} for some $r<\alpha,\beta<\infty$. In fact, the second equation in \eqref{GrindEQ__4_12_} implies that
\begin{equation}\label{GrindEQ__4_15_}
\frac{1}{p'}<\frac{4}{d}+\frac{\sigma+1}{r}-\frac{\sigma}{2}.
\end{equation}
Since $r<\alpha,\beta<\infty$, it follows from the first equation in \eqref{GrindEQ__4_12_} that
\begin{equation}\label{GrindEQ__4_16_}
\max\left\{\frac{1}{r}+\frac{b}{d},~\frac{b+s}{d}\right\}<\frac{1}{p'}<\frac{\sigma+1}{r}+\frac{b}{d}.
\end{equation}
Hence, it suffices to prove that there exist $B$-admissible pairs $(\gamma(p),p)$ and $(\gamma(r),r)$ satisfying \eqref{GrindEQ__4_15_} and \eqref{GrindEQ__4_16_}.

First, we consider the case $d\le 4$. In this case, we can see that there exits a $B$-admissible pair $(\gamma(p),p)$ satisfying \eqref{GrindEQ__4_15_} and \eqref{GrindEQ__4_16_} provided that
\begin{equation}\label{GrindEQ__4_17_}
b<d-s,~\frac{1}{r}<1-\frac{b}{d}
\end{equation}
and
\begin{equation}\label{GrindEQ__4_18_}
\frac{\sigma+1}{r}>\max\left\{\frac{1}{2}-\frac{b}{d},~\frac{\sigma+1}{2}-\frac{4}{d},~
\frac{\sigma}{2}+\frac{b+s-4}{d},~\frac{\sigma+1}{\sigma}\left(\frac{\sigma}{2}+\frac{b-4}{d}\right)\right\}.
\end{equation}
Using the fact that $b<d-s$ and $s\ge \frac{d}{2}$, we can verify that there exists a $B$-admissible pair $(\gamma(r),r)$ satisfying \eqref{GrindEQ__4_17_} and \eqref{GrindEQ__4_18_}.

Next, we consider the case $d\ge 5$. In this case, we can see that there exits a $B$-admissible pair $(\gamma(p),p)$ satisfying \eqref{GrindEQ__4_15_} and \eqref{GrindEQ__4_16_} provided that
\begin{equation}\label{GrindEQ__4_19_}
b<2+\frac{d}{2}-s,~\frac{1}{r}<\frac{d+4-2b}{2d}
\end{equation}
and
\begin{equation}\label{GrindEQ__4_20_}
\frac{\sigma+1}{r}>\max\left\{\frac{1}{2}-\frac{b}{d},~\frac{\sigma+1}{2}-\frac{4}{d},~
\frac{\sigma}{2}+\frac{b+s-4}{d},~\frac{\sigma+1}{\sigma}\left(\frac{\sigma}{2}+\frac{b-4}{d}\right)\right\}.
\end{equation}
Using the fact that $b<2+\frac{d}{2}-s$ and $s\ge \frac{d}{2}$, we can easily see that there exists a $B$-admissible pair $(\gamma(r),r)$ satisfying \eqref{GrindEQ__4_19_} and \eqref{GrindEQ__4_20_}. This completes the proof.
\end{proof}

Now we are ready to prove Theorem \ref{thm 1.1.}.
\begin{proof}[{\bf Proof of Theorem \ref{thm 1.1.}}]
Since the proof is standard, we only sketch the proof. Let $T>0$ and $M>0$ which will be chosen later. Given $I=[-T,\;T]$, we define
\begin{equation} \nonumber
D=\left\{u\in L^{\gamma (r)} (I,\;H^{s}_{r,2}):\;\left\| u\right\| _{L^{\gamma (r)} \left(I,\;H^{s}_{r,2}\right)} \le M\right\},
\end{equation}
where $(\gamma(r),r)$ is as in Lemma \ref{lem 4.3.}.
Putting
\begin{equation} \nonumber
d\left(u,\;v\right)=\;\left\| u-v\right\| _{L^{\gamma(r)} (I,\;L^{r,2})},
\end{equation}
$(D,d)$ is a complete metric space (see \cite{AT21}).
Now we consider the mapping
\begin{equation} \label{GrindEQ__4_21_}
G:\;u(t)\to S(t)u_{0} -i\lambda \int _{0}^{t}S(t-\tau)|x|^{-b} |u(\tau)|^{\sigma}u(\tau)d\tau
\equiv u_{L} +u_{NL} ,
\end{equation}
where
\begin{equation} \nonumber
u_{L} =S(t)u_{0} ,~u_{NL} =-i\lambda \int _{0}^{t}S(t-\tau )|x|^{-b}|u(\tau)|^{\sigma}u(\tau)d\tau  .
\end{equation}
Lemma \ref{lem 4.1.} (Strichartz estimates) yields that
\begin{equation}\label{GrindEQ__4_22_}
\left\| u_{NL} \right\| _{L^{\gamma (r)} (I,\;H^{s}_{r,2})} \lesssim\left\| |x|^{-b}|u|^{\sigma}u\right\|
_{L^{\gamma (p)' } (I,\;H^{s}_{p',2})},
\end{equation}
\begin{equation}\label{GrindEQ__4_23_}
\left\| Gu-Gv \right\| _{L^{\gamma (r)} (I,\;L^{r,2})} \lesssim\left\| |x|^{-b} \left(|u|^{\sigma}u-|v|^{\sigma}v\right)\right\|_{L^{\gamma (p' } (I,\;L^{p',2})}.
\end{equation}
Using Lemma \ref{lem 4.1.}, \eqref{GrindEQ__4_22_}, \eqref{GrindEQ__4_23_} and the standard contraction argument (see e.g. \cite{AT21} or Section 4.9 in \cite{C03}), we can get the desired results and we omit the details. This concludes the proof.
\end{proof}

\subsection{Small data global well-posedness}
In this subsection, we prove Theorem \ref{thm 1.3.}.

\begin{lemma}\label{lem 4.4.}
Let $0\le s< \min\{2+\frac{d}{2},\;d\}$, $0<b<\min\{4,\;2+\frac{d}{2}-s,\;d-s\}$ and $\frac{8-2b}{d}\le \sigma\le \sigma_{c}(s)$ with $\sigma<\infty$. If $\sigma$ is not an even integer, assume that $\sigma>\left\lceil s\right\rceil -1$. Then for any interval $I(\subset \mathbb R)$, there exist $B$-admissible pairs $(\gamma(\bar{p}),\bar{p})$ and $(\gamma(\bar{r}),\bar{r})$ such that
\begin{equation} \label{GrindEQ__4_24_}
\left\||u|^{\sigma}u\right\|_{L^{\gamma(\bar{p})'}(I,\; \dot{H}_{\bar{p}',2}^{s})}
\lesssim \left\|u\right\|^{\sigma}_{L^{\gamma(\bar{r})}(I,\; \dot{H}_{\bar{r},2}^{s_{c}}\cap \dot{H}_{\bar{r},2}^{\tilde{s}_{c}})}\left\|u\right\|_{L^{\gamma(\bar{r})}(I,\; H_{\bar{r},2}^{s})},
\end{equation}
\begin{equation} \label{GrindEQ__4_25_}
\left\||u|^{\sigma}u\right\|_{L^{\gamma(\bar{p})'}(I,\; \dot{H}_{\bar{p}',2}^{s_{c}})\cap L^{\gamma(\bar{p})'}(I,\; \dot{H}_{\bar{p}',2}^{\tilde{s}_{c}})}
\lesssim \left\|u\right\|^{\sigma+1}_{L^{\gamma(\bar{r})}(I,\; \dot{H}_{\bar{r},2}^{s_{c}}\cap \dot{H}_{\bar{r},2}^{\tilde{s}_{c}})},
\end{equation}
\begin{equation} \label{GrindEQ__4_26_}
\left\||u|^{\sigma}v\right\|_{L^{\gamma(\bar{p})'}(I,\; L^{\bar{p}',2})}
\lesssim \left\|u\right\|^{\sigma}_{L^{\gamma(\bar{r})}(I,\; \dot{H}_{\bar{r},2}^{s_{c}})}\left\|u\right\|_{L^{\gamma(\bar{r})}(I, \;L^{\bar{r},2})},
\end{equation}
where $s_{c}$ and $\tilde{s}_{c}$ are given in \eqref{GrindEQ__1_11_}.
\end{lemma}
\begin{proof}
Using the fact $\frac{8-2b}{d}\le \sigma\le \sigma_{c}(s)$, we can easily see that
\begin{equation} \label{GrindEQ__4_27_}
0\le s_{c}\le \tilde{s}_{c}\le s.
\end{equation}
Let $(\gamma(\bar{p}),\bar{p})$ and $(\gamma(\bar{r}),\bar{r})$ satisfy
\begin{equation} \label{GrindEQ__4_28_}
\frac{1}{\gamma(\bar{p})'}=\frac{\sigma+1}{\gamma(\bar{r})},~\frac{1}{\bar{r}}>\frac{\tilde{s}_{c}}{d}.
\end{equation}
Using \eqref{GrindEQ__4_28_}, we can see that
\begin{equation} \label{GrindEQ__4_29_}
\frac{1}{\bar{p}'}=\sigma\left(\frac{1}{\bar{r}} -\frac{s_{c}}{d}\right) +\frac{1}{\bar{r}}+\frac{b}{d}=(\sigma +1)\left(\frac{1}{\bar{r}} -\frac{\tilde{s}_{c}}{d}\right)+\frac{b+s}{d},
\end{equation}
Putting $\frac{1}{\alpha_{1}}:=\frac{1}{\bar{r}}-\frac{s_{c}}{d}$ and $\frac{1}{\alpha_{2}}:=\frac{1}{\bar{r}}-\frac{\tilde{s}_{c}}{d}$, it follows from \eqref{GrindEQ__4_27_}, \eqref{GrindEQ__4_28_} and Lemma \ref{lem 3.5.} that
\begin{equation} \label{GrindEQ__4_30_}
\dot{H}^{s_{c}}_{\bar{r},2}\hookrightarrow L^{\alpha_{1},2},~\dot{H}^{\tilde{s}_{c}}_{\bar{r},2}\hookrightarrow L^{\alpha_{2},2}.
\end{equation}
Using \eqref{GrindEQ__4_29_}, \eqref{GrindEQ__4_30_}, Lemmas \ref{lem 2.1.}--\ref{lem 2.4.}, \ref{lem 3.4.}, \ref{lem 3.6.}, \ref{lem 3.8.}, Corollary \ref{cor 3.13.} and Remark \ref{rem 4.2.}, we have
\begin{eqnarray}\begin{split} \label{GrindEQ__4_31_}
\left\| |x|^{-b}|u|^{\sigma}u\right\| _{\dot{H}_{\bar{p}',2}^{s} } &\lesssim \left\| |x|^{-b}\right\| _{L^{\bar{p}_{1},\infty}}\left\| |u|^{\sigma}u\right\| _{\dot{H}_{\bar{p}_{2},2}^{s}}
+\left\| |x|^{-b}\right\| _{\dot{H}_{\bar{p}_{3},\infty}^{s}}\left\| |u|^{\sigma}u\right\| _{L^{\bar{p}_{4},2}}\\
&\lesssim \left\| |u|^{\sigma}u\right\| _{\dot{H}_{\bar{p}_{2},2}^{s}}+\left\| |u|^{\sigma}u\right\| _{L^{\bar{p}_{4},2}}
\lesssim \left\| u\right\| _{L^{\alpha_{1},2}}^{\sigma}\left\| u\right\| _{\dot{H}_{\bar{r},2}^{s}}+\left\| u\right\| _{L^{\alpha_{2},2}}^{\sigma+1}\\
&\lesssim \left\| u\right\| _{\dot{H}_{\bar{r},2}^{s_{c}} }^{\sigma}\left\| u\right\| _{\dot{H}_{\bar{r},2}^{s}}+\left\| u\right\| _{\dot{H}_{\bar{r},2}^{\tilde{s}_{c}} }^{\sigma+1}
\lesssim \left\|u\right\|^{\sigma}_{\dot{H}_{\bar{r},2}^{s_{c}}\cap \dot{H}_{\bar{r},2}^{\tilde{s}_{c}}}\left\|u\right\|_{H_{\bar{r},2}^{s}}
\end{split}\end{eqnarray}
where
\begin{equation}\label{GrindEQ__4_32_}
\frac{1}{\bar{p}_{1}}:=\frac{b}{d},~\frac{1}{\bar{p}_{2}}:=\frac{1}{p'}-\frac{b}{d},~\frac{1}{\bar{p}_{3}}:=\frac{b+s}{d},~\frac{1}{\bar{p}_{4}}:=\frac{1}{\bar{p}'}-\frac{b+s}{d}.
\end{equation}
Putting
\begin{equation}\label{GrindEQ__4_33_}
\frac{1}{\bar{p}_{5}}:=\frac{b+s_{c}}{d},~\frac{1}{\bar{p}_{6}}:=(\sigma +1)\left(\frac{1}{\bar{r}} -\frac{s_{c}}{d}\right),~\frac{1}{\bar{p}_{7}}:=\frac{b+\tilde{s}_{c}}{d},~\frac{1}{\bar{p}_{8}}:=\sigma\left(\frac{1}{\bar{r}} -\frac{s_{c}}{d}\right)+\frac{1}{\bar{r}} -\frac{\tilde{s}_{c}}{d},
\end{equation}
it follow from \eqref{GrindEQ__4_29_} that $\frac{1}{\bar{p}'}=\frac{1}{\bar{p}_{5}}+\frac{1}{\bar{p}_{6}}=\frac{1}{\bar{p}_{7}}+\frac{1}{\bar{p}_{8}}$.
Hence, using the similar argument as in the proof of \eqref{GrindEQ__4_31_}, we also have
\begin{eqnarray}\begin{split} \label{GrindEQ__4_34_}
\left\| |x|^{-b}|u|^{\sigma}u\right\| _{\dot{H}_{\bar{p}',2}^{s_{c}} } &\lesssim \left\| |x|^{-b}\right\| _{L^{\bar{p}_{1},\infty}}\left\| |u|^{\sigma}u\right\| _{\dot{H}_{\bar{p}_{2},2}^{s_{c}}}
+\left\| |x|^{-b}\right\| _{\dot{H}_{\bar{p}_{5},\infty}^{s_{c}}}\left\| |u|^{\sigma}u\right\| _{L^{\bar{p}_{6},2}}\\
&\lesssim \left\| u\right\| _{L^{\alpha_{1},2}}^{\sigma}\left\| u\right\| _{\dot{H}_{\bar{r},2}^{s_{c}}}+\left\| u\right\| _{L^{\alpha_{1},2}}^{\sigma+1}
\lesssim \left\| u\right\| _{\dot{H}_{\bar{r},2}^{s_{c}} }^{\sigma +1},
\end{split}\end{eqnarray}
\begin{eqnarray}\begin{split} \label{GrindEQ__4_35_}
\left\| |x|^{-b}|u|^{\sigma}u\right\| _{\dot{H}_{\bar{p}',2}^{\tilde{s}_{c}} } &\lesssim \left\| |x|^{-b}\right\| _{L^{\bar{p}_{1},\infty}}\left\| |u|^{\sigma}u\right\| _{\dot{H}_{\bar{p}_{2},2}^{\tilde{s}_{c}}}
+\left\| |x|^{-b}\right\| _{\dot{H}_{\bar{p}_{7},\infty}^{\tilde{s}_{c}}}\left\| |u|^{\sigma}u\right\| _{L^{\bar{p}_{8},2}}\\
&\lesssim \left\| u\right\| _{L^{\alpha_{1},2}}^{\sigma}\left\| u\right\| _{\dot{H}_{\bar{r},2}^{\tilde{s}_{c}}}+\left\| u\right\| _{L^{\alpha_{1},2}}^{\sigma}\left\| u\right\| _{L^{\alpha_{2},2}}^{\sigma}\\
&\lesssim \left\| u\right\| _{\dot{H}_{\bar{r},2}^{s_{c}} }^{\sigma }\left\| u\right\| _{\dot{H}_{\bar{r},2}^{\tilde{s}_{c}}},
\end{split}\end{eqnarray}
and
\begin{eqnarray}\begin{split} \label{GrindEQ__4_36_}
\left\| |x|^{-b}|u|^{\sigma}v\right\| _{L^{\bar{p}',2}} &\lesssim  \left\| |x|^{-b}\right\| _{L^{\bar{p}_{1},\infty}}\left\| |u|^{\sigma}v\right\| _{L^{\bar{p}_{2},2}}
\lesssim \left\| u\right\| _{L^{\alpha_{1},2}}^{\sigma}\left\| v\right\| _{L^{\bar{r},2}}
\lesssim \left\| u\right\| _{\dot{H}_{\bar{r},2}^{s_{c}} }^{\sigma }\left\| v\right\| _{L^{\bar{r},2}}
\end{split}\end{eqnarray}
Using \eqref{GrindEQ__4_28_}, \eqref{GrindEQ__4_31_}, \eqref{GrindEQ__4_34_}--\eqref{GrindEQ__4_36_} and H\"{o}lder inequality, we immediately get \eqref{GrindEQ__4_24_}-- \eqref{GrindEQ__4_26_}. So it suffices to prove that there exist $B$-admissible pairs $(\gamma(\bar{p}),\bar{p})$ and $(\gamma(\bar{r}),\bar{r})$ satisfying \eqref{GrindEQ__4_28_}.

If $d\ge5$, then we put $\bar{p}:=\frac{2d}{d-4}$ and $\frac{1}{\bar{r}}:=\frac{1}{2}-\frac{2}{d(\sigma+1)}$. One can easily check that $(\gamma(\bar{p}),\bar{p})$ and $(\gamma(\bar{r}),\bar{r})$ are $B$-admissible and satisfy \eqref{GrindEQ__4_28_} by using the fact $b<2+\frac{d}{2}-s$.

It remains to consider the case $d\le 4$. Since $b<d-s$, there exists $\bar{p}$ large enough such that
$$
\frac{1}{\bar{p}}<1-\frac{b+s}{d},~2<\bar{p}<\infty.
$$
It is obvious that $(\gamma(\bar{p}),\bar{p})$ is $B$-admissible. \eqref{GrindEQ__4_28_} implies that
$$
\frac{1}{\bar{r}}=\frac{1}{2}-\frac{1}{\sigma+1}\left(\frac{4}{d}-\frac{1}{2}+\frac{1}{\bar{p}}\right).
$$
We can easily verify that $(\gamma(\bar{r}),\bar{r})$ is $B$-admissible by using the fact $\sigma\ge \frac{8-2b}{d}$ and $b<d-s$.
Furthermore, we can easily check that $\frac{1}{\bar{r}}>\frac{\tilde{s}_{c}}{d}$ is equivalent to $b<d-s$.
This complete the proof.
\end{proof}

We are ready to prove Theorem \ref{thm 1.3.}.
\begin{proof}[{\bf Proof of Theorem \ref{thm 1.3.}}]
Let $\bar{M}>0$ and $\bar{m}>0$, which will be chosen later. For $(\gamma(\bar{r}),\bar{r})$ given in Lemma \ref{lem 4.4.}, we define the following complete metric space
\begin{equation} \nonumber
\bar{D}=\left\{u\in L^{\gamma (p)} \left(\mathbb R,\;H_{r,2}^{s} \right):\;\left\| u\right\| _{L^{\gamma (\bar{r})} \left(\mathbb R,\;\dot{H}_{\bar{r},2}^{s_{c}}\cap \dot{H}_{\bar{r},2}^{\tilde{s}_{c}} \right)} \le \bar{m},\;\left\| u\right\| _{L^{\gamma (\bar{r})} \left(\mathbb R,\;H_{\bar{r},2}^{s} \right)} \le \bar{M}\right\},
\end{equation}
where
which is equipped with the metric
\begin{equation} \nonumber
d\left(u,\;v\right)=\left\| u-v\right\| _{L^{\gamma (\bar{r})} \left(\mathbb R,\;L^{\bar{r},2} \right)} .
\end{equation}
Using \eqref{GrindEQ__4_21_}, Lemma \ref{lem 4.1.} (Strichartz estimates), Lemma \ref{lem 4.4.} and the fact
\begin{equation} \nonumber
\left||x|^{-b} |u|^{\sigma}u-|x|^{-b} |v|^{\sigma}v\right|\lesssim|x|^{-b}(\left|u\right|^{\sigma } +\left|v\right|^{\sigma })|u-v|,
\end{equation}
we have
\begin{equation} \label{GrindEQ__4_37_}
\left\| Gu\right\| _{L^{\gamma (\bar{r})} (\mathbb R,\;H_{\bar{r},2}^{s} )} \le C(\left\| u_{0} \right\| _{H^{s} } +\left\|u\right\|^{\sigma}_{L^{\gamma(\bar{r})}(\mathbb R,\; \dot{H}_{\bar{r},2}^{s_{c}}\cap \dot{H}_{\bar{r},2}^{\tilde{s}_{c}})}\left\|u\right\|_{L^{\gamma(\bar{r})}(\mathbb R,\; H_{\bar{r},2}^{s})} ),
\end{equation}
\begin{equation} \label{GrindEQ__4_38_}
\left\| Gu\right\| _{L^{\gamma(\bar{r})}(\mathbb R,\; \dot{H}_{\bar{r},2}^{s_{c}}\cap \dot{H}_{\bar{r},2}^{\tilde{s}_{c}})}\le C(\left\| u_{0} \right\| _{\dot{H}^{s_{c}}\cap \dot{H}^{\tilde{s}_{c}}} +\left\|u\right\|^{\sigma+1}_{L^{\gamma(\bar{r})}(\mathbb R,\; \dot{H}_{\bar{r},2}^{s_{c}}\cap \dot{H}_{\bar{r},2}^{\tilde{s}_{c}})} ),
\end{equation}
\begin{equation} \label{GrindEQ__4_39_}
\left\| Gu-Gv\right\| _{L^{\gamma (\bar{r})} \left(\mathbb R,\;L^{\bar{r},2} \right)} \le C(\left\|u\right\|^{\sigma}_{L^{\gamma(\bar{r})}(\mathbb R,\; \dot{H}_{\bar{r},2}^{s_{c}})}+\left\|v\right\|^{\sigma}_{L^{\gamma(\bar{r})}(\mathbb R,\; \dot{H}_{\bar{r},2}^{s_{c}})})\left\|u-v\right\|_{L^{\gamma(\bar{r})}(\mathbb R, \;L^{\bar{r},2})}.
\end{equation}
Put $m=2C\left\| u_{0} \right\| _{\dot{H}^{s_{c}}\cap \dot{H}^{\tilde{s}_{c}}} $, $M=2C\left\| u_{0} \right\| _{H^{s} } $ and $\delta =2\left(4C\right)^{-\frac{\sigma +1}{\sigma } } $. If $\left\| u_{0} \right\| _{\dot{H}^{s_{c}}\cap \dot{H}^{\tilde{s}_{c}}} \le \delta $, i.e. $Cm^{\sigma } <\frac{1}{4} $, then it follows from \eqref{GrindEQ__4_37_}--\eqref{GrindEQ__4_39_} that
\begin{equation} \label{GrindEQ__4_40_}
\left\| Gu\right\| _{L^{\gamma (\bar{r})} (\mathbb R,\;H_{\bar{r},2}^{s} )} \le \frac{m}{2} +Cm^{\sigma } M\le M,
\end{equation}
\begin{equation} \label{GrindEQ__4_41_}
\left\| Gu\right\| _{L^{\gamma(\bar{r})}(\mathbb R,\; \dot{H}_{\bar{r},2}^{s_{c}}\cap \dot{H}_{\bar{r},2}^{\tilde{s}_{c}})} \le \frac{m}{2} +Cm^{\sigma +1} \le m,
\end{equation}
\begin{equation} \label{GrindEQ__4_42_}
\left\| Gu-Gv\right\| _{L^{\gamma (\bar{r})} \left(\mathbb R,\;L^{\bar{r},2} \right)} \le 2Cm^{\sigma } \left\|u-v\right\|_{L^{\gamma(\bar{r})}(\mathbb R, \;L^{\bar{r},2})}\le \frac{1}{2} \left\|u-v\right\|_{L^{\gamma(\bar{r})}(\mathbb R, \;L^{\bar{r},2})}.
\end{equation}
Hence, $G:(\bar{D},\;d)\to (\bar{D},\;d)$ is a contraction mapping and there exists a unique solution of \eqref{GrindEQ__1_1_} in $\bar{D}$.
It remains to prove the scattering result. We can see that \eqref{GrindEQ__1_13_} is equivalent to
\[{\mathop{\lim }\limits_{t\to \pm \infty }} \left\| e^{-it\Delta^{2} } u(t)-u_{0}^{\pm } \right\| _{H^{s}} =0.\]
In other words, it suffices to show that $e^{-it\Delta^{2} } u(t)$ converges in $H^{s} $ as $t_{1} ,\;t_{2} \to \pm \infty $.
Let $0<t_{1} <t_{2} <+\infty $. By using Lemma \ref{lem 4.1.} (Strichartz estimates) and Lemma \ref{lem 4.4.}, we have
\begin{eqnarray}\begin{split} \label{GrindEQ__4_43_}
&\left\| e^{-it_{2} \Delta^{2}} u(t_{2})-e^{-it_{1} \Delta^{2} } u(t_{1})\right\| _{H^{s} }=\left\| \int _{t_{1} }^{t_{2} }e^{-i\tau \Delta^{2} } |x|^{-b} |u(\tau)|^{\sigma}u(\tau) d\tau  \right\| _{H^{s} } \\
&~~~~~~~~~~~~~~~~~~~\lesssim\left\| |x|^{-b} |u|^{\sigma}u\right\| _{L^{\gamma (\bar{p})'} (\left(t_{1} ,\; t_{2} \right),\;H_{\bar{p}',2}^{s} )} \lesssim\left\| u\right\| _{L^{\gamma (\bar{r})} (\left(t_{1} ,\; t_{2} \right),\;H_{\bar{r},2}^{s})}^{\sigma +{\rm 1}}.
\end{split}\end{eqnarray}
Using \eqref{GrindEQ__4_43_} and the fact $\left\| u\right\| _{L^{\gamma (\bar{r})} \left(\mathbb R,\;H_{\bar{r},2}^{s} \right)} <\infty $, we have
\[\left\| e^{-it_{2} \Delta^{2} } u(t_{2})-e^{-it_{1} \Delta^{2} } u(t_{1})\right\| _{H^{s}} \to 0,\]
as $t_{1} ,\;t_{2} \to +\infty $. Thus, the limit $u_{0}^{+} :={\mathop{\lim }\limits_{t\to +\infty }} e^{-it\Delta^{2} } u\left(t\right)$
exits in $H^{s}$. This shows the small data scattering for positive time, the one for negative time is treated similarly. This completes the proof.
\end{proof}


\end{document}